\documentclass[twoside,11pt]{article} 

\usepackage{fullpage,graphicx,amsmath,amssymb,bbm,amsfonts,amsthm,enumerate,natbib,url,color,caption,enumerate}
\usepackage[font = scriptsize]{subcaption}

\theoremstyle{plain}
\newtheorem{thm}{Theorem}
\newtheorem{lem}[thm]{Lemma}
\newtheorem{prop}[thm]{Proposition}
\newtheorem{cor}[thm]{Corollary}

\theoremstyle{definition}

\theoremstyle{remark}

\numberwithin{equation}{section}
\newcommand{\Step}[1]{\textit{Step #1:}} 
\newcommand{\abs}[1]{\left| #1\right|}
\newcommand{\norm}[1]{\left\lVert #1 \right\rVert}
\newcommand{\norms}[1]{\| #1 \|} 
\newcommand{\ind}{\mathbbm{1}}
\newcommand{\pr}{\mathbb{P}}
\newcommand{\R}{\mathbb{R}}
\newcommand{\E}{\mathbb{E}}

\newcommand{\floor}[1]{\left\lfloor #1 \right\rfloor}
\newcommand{\ceil}[1]{\left\lceil #1 \right\rceil}
\newcommand{\Cov}{\mathrm{Cov}}
\newcommand{\Var}{\mathrm{Var}}
\newcommand{\Tr}{\mathrm{Tr}}

\newcommand{\argmin}[1]{\underset{#1}{\operatorname{arg}\operatorname{min}}\;}

\newcommand{\indist}{\stackrel{d}{\to}}

\newcommand{\mb}{\mathbf}
\newcommand{\mbb}{\boldsymbol}
\newcommand{\MSPE}{\mathrm{MSPE}}

\title{On $b$-bit min-wise hashing for large-scale regression and classification with sparse data}
\author{Rajen D. Shah\thanks{Supported by The Alan Turing Institute under the EPSRC grant EP/N510129/1 and an EPSRC programme grant.} \\ 
University of Cambridge \\
\url{r.shah@statslab.cam.ac.uk}
\and Nicolai Meinshausen \\ 
ETH Z\"urich\\
\url{meinshausen@stat.math.ethz.ch}}

\begin{document}
\maketitle
\begin{abstract}
Large-scale regression problems where both the number of variables, $p$, and the number of observations, $n$, may be large and in the order of millions or more, are becoming increasingly more common. Typically the data are sparse: only a fraction of a percent of the entries in the design matrix are non-zero. Nevertheless, often the only computationally feasible approach is to perform dimension reduction to obtain a new design matrix with far fewer columns and then work with this compressed data.

$b$-bit min-wise hashing \citep{li2011theory,li2011hashing} is a promising  dimension reduction scheme for sparse matrices
which produces a set of random features such that regression on the resulting design matrix approximates a kernel regression with the resemblance kernel. In this work, we derive bounds on the prediction error of such regressions.
For both linear and logistic models, we show that the average prediction error vanishes asymptotically as long as $q \|\mbb\beta^*\|_2^2 /n \rightarrow 0$, where $q$ is the average number of non-zero entries in each row of the design matrix and $\mbb \beta^*$ is the coefficient of the linear predictor.

We also show that ordinary least squares or ridge regression applied to the reduced data can in fact allow us fit more flexible models. We obtain non-asymptotic prediction error bounds for interaction models and for models where an unknown row normalisation must be applied in order for the signal to be linear in the predictors.
\end{abstract}

\bigskip
  

\section{Introduction}
The modern field of high-dimensional statistics has now developed a powerful range of methods to deal with datasets where the number of variables $p$ may greatly exceed the number of variables~$n$ (see 
 \citet{buhlmann2011statistics} for an overview of recent advances). The prototypical example of microarray data, where $p$ may be in the tens of thousands but~$n$ is typically not more than a few hundred, has motivated much of this development. Yet not all modern datasets come in this sort of shape and size. The emerging area of `large-scale data' or the more vaguely defined `Big Data' is a response to the increasing prevalence of computationally challenging datasets as arise in text analysis or web-scale prediction tasks, to give two examples. Here both $n$ and $p$ can run into the millions or more, particularly if interactions are considered. In these `large $p$, large $n$' regression scenarios, one can imagine situations where ordinary least squares (OLS) has a  competitive performance for prediction, but the sheer size of the data renders it infeasible for computational rather than statistical reasons.

An important feature of many large-scale datasets is that they are sparse: the overwhelming majority of entries in the design matrices are exactly zero. This is not to be confused with signal sparsity, a common assumption in the high-dimensional context. Indeed, when the design matrix is sparse, having only a few variables that contribute to the response would make the expected response values of all observations with no non-zero entries for the important variables exactly the same; one expects that such a property would not be possessed by many datasets. However, similarly to the way in which many high-dimensional techniques exploit sparsity to improve statistical efficiency, one might hope that sparsity in the data could be leveraged to yield both computational and statistical improvements, and indeed we demonstrate in this work that this can be achieved.

Kernel machines are an important class of machine learning methods for which such large-scale data poses particularly serious computational challenges. For example, standard implementations of kernel ridge regression would have computational complexity $O(n^3)$ and a storage cost of $O(n^2)$ when $p$ is considered fixed; a large $p$ will increase these computational costs depending on the kernel to be used. There has therefore been a great deal of work on approximating kernel machines by first randomly mapping the $n\times p$ design matrix $\mb X$ to a $n \times d$ matrix $\mb S$ with $d \ll p$ such that dot products between rows of $\mb S$ approximate the kernel evaluated on the corresponding rows of $\mb X$. Then a regular ridge regression on $\mb S$ will resemble a kernel ridge regression on $\mb X$, for example.

A remarkably effective way of forming $\mb S$ that is applicable when the design matrix is sparse and binary, is $b$-bit min-wise hashing \citep{li2011theory,li2011hashing} which is based on an earlier technique called min-wise hashing \citep{broder1998min,cohen2001finding,datar2002estimating}. Here $\mb S$ is constructed such that the dot product between any two rows of $\mb S$, $\mb s_i^T \mb s_j$, can approximate the \emph{resemblance} or \emph{Jaccard similarity} or between the corresponding rows of $\mb X$, defined as $|\mb z_i \cap \mb z_j|/|\mb z_i \cup \mb z_j| $
where $\mb z_i = \{k : X_{ik} \neq 0\}$.

The empirical performance of regression and classification procedures following $b$-bit min-wise hashing \citep{li2011hashing, li2013b} is particularly impressive. Existing theory on $b$-bit min-wise hashing \citep{li2011theory} has focused on the variance and bias in the approximation of the kernel. However, there remain significant gaps in our theoretical understanding of this important procedure when used to approximate a kernel machine:
\begin{enumerate}[(a)]
\item What sorts of regression models is the resemblance kernel well-suited for and how does sparsity of the design matrix play a role?
\item What is the loss in prediction accuracy due to the approximation provided by $b$-bit min-wise hashing for different sorts of regression procedures?
\item What is the overall prediction error incurred by different regression methods following $b$-bit min-wise hashing in different regression models?
\end{enumerate}
An answer to (c) would be the ultimate goal here, and it would appear that in order to tackle this one must first solve (a) and (b). In this paper, we take a very different approach and aim to answer (c) directly: rather than considering what sorts of functions lie in the reproducing kernel Hilbert space (RKHS) associated with the resemblance kernel and have low RKHS norms, we look at the sorts of signals that can be approximated well by linear combinations of columns of the matrix $\mb S$ constructed by $b$-bit min-wise hashing. In this way, we use the random feature expansions provided by $b$-bit min-wise hashing to understand the predictive properties of the resemblance kernel.

\subsection{Our contributions and organisation of the paper}
In this paper we derive finite-sample bounds on the expected risk of linear and logistic regression following dimension reduction through $b$-bit min-wise hashing under various different models. Our results show that the method, and hence also the resemblance kernel, are particularly suited to sparse data.

We describe the $b$-bit min-wise hashing algorithm in Section~\ref{sec:MRS} and also discuss in greater details the connection to the resemblance kernel. We also introduce a generalisation of $b$-bit min-wise hashing applicable to sparse data with real-valued entries motivated by our theory.
Perhaps the simplest sorts of signals that we could hope to be able to fit well are linear signals of the form $\mb X \mbb\beta^*$.
In Section~\ref{section:main} we first consider how well a linear combination of columns of $\mb S$ can approximate such a signal. We then study a much larger class of signals defined by first scaling the rows of $\mb X$ in different ways depending on their sparsity and then forming a linear signal from a scaled version of $\mb X$. Some form of row normalisation is often performed on the original data as a pre-processing step, but the optimal normalisation to use is seldom known; our theory shows how $b$-bit min-wise hashing, and hence also the resemblance kernel, is able to automatically discover an appropriate scaling in several settings.

In Section~\ref{sec:OLS_main} we study the performance of ordinary least squares, ridge regression and $\ell_2$-penalised logistic regression using the reduced design matrix it creates. Our results are applicable to both linear signals and nonlinear signals of the sort described above. In the former setting, we show that the expected mean-squared prediction error is bounded by a small constant times $\sqrt{q/n} \norm{\mbb \beta^*}_2$, where $q$ is the average number nonzero entries in the rows of $\mb X$ and $\mbb\beta^*$ is the coefficient vector.
We present similar results for logistic regression. 

In Section~\ref{section:interaction} we study another form of nonlinear signal that can be approximated by the $b$-bit minwise hashing and the resemblance kernel: we show that interaction models in the original data can also be captured by main effects regression on the compressed data.
Variable importance measures are discussed in Section~\ref{section:extensions}. We conclude with a discussion in Section~\ref{sec:discussion}. The appendix contains all proofs, an additional result concerning the implications of our approximation error bound for properties of the RKHS of the resemblance kernel, and an empirical study validating our bounds.

\subsection{Related work}
There has been very little work in understanding properties of the resemblance kernel. One of the few pieces of work in this direction is  \citet{BOUCHARD2013615}, who show that the kernel matrix with entries given by the Jaccard similarity between different elements of the power set of $\{1,\ldots,p\}$ minus the empty set is positive definite. It follows that the RKHS of the resemblance kernel contains every real-valued function on $p$-dimensional binary vectors (see Section~\ref{sec:RKHS_interpret}). However, this result is not informative for understanding which sorts of regression models a kernel ridge regression will perform well for, a question which we provide some answers to through our study of $b$-bit min-wise hashing.

Approximating kernel methods using random feature expansions was pioneered by \citet{Rahimi2007} who used random Fourier features to approximate translation invariant kernels such as the Gaussian kernel. \citet{sutherland2015} provides bounds on the approximation of the corresponding kernel as well as bounds on the distance between the predictions from regression on the random features and kernel ridge regression in terms of distances between the true kernel and its approximation.
 \citet{Le2013} introduce a scheme related to random Fourier features that further improves the computational efficiency. \citet{rahimi2008uniform} consider more general random feature expansions and study how well they can approximate functions in a family determined by the distribution of feature expansions in terms of a certain form of function norm defined on the family. \citet{rahimi2009weighted} provides prediction error bounds for a method that minimises the empirical risk of a weighted sum of random feature expansions where weights are constrained in $\ell_\infty$-norm.
\citet{bach2015equivalence} studies how well random feature expansions can approximate elements of their corresponding RKHS in terms of the eigenvalues of the associated kernel integral operator. The Nystr\"om method \citep{williams2001using} is related and aiming at a computationally efficient low-rank approximation to the full kernel matrix; see  \citep{bach2013sharp} and \citep{rudi2015less} for approximation guarantees.

A distinguishing feature of our work is that bounds are obtained not in terms of the norm of the RKHS of the resemblance kernel, which would be difficult to interpret, but in terms of quantities derived directly from the different models considered (we look at linear models with unknown row scaling and at nonlinear interaction models). We could divide the analysis into two parts: (i) first we could try to understand the predictive accuracy when using exact kernel regression with the resemblance kernel for such true regression functions and then (ii) in a second step understand how much predictive accuracy we lose by using $b$-bit minwise hashing as an approximation to using exact kernel regression with the resemblance kernel. Instead of making these two separate steps, we study here directly how well $b$-bit minwise hashing performs for these model classes. 

Properties of $b$-bit min-wise hashing related to similarity search are studied in \citet{li2011theory}. Theory concerning its use for large-scale learning is presented in \citet{li2011hashing} which quantifies the mean and variance of entries in the Gram matrix $\mb S \mb S^T$ and its relationship to the resemblance kernel as well as providing comparisons with random projections and \emph{Vowpal Wabbit}. 
Random feature expansions for other types of kernels are developed in \citet{shi2009hash, weinberger2009feature, Vedaldi2012, AISTATS2012_KarK12, Li2014, pennington2015spherical}.
 
More generally, there is a huge variety of dimension reduction schemes across the statistics and computer science literature.
Performing principal component analysis \citep{jolliffe1986principal} (PCA) and retaining only the first $d$ components is one of the most popular methods. One drawback however in the large-scale data setting is that computing the principal components can be computationally demanding.
The method of random projections, motivated by the celebrated Johnson--Lindenstrauss lemma \citep{johnson1984extensions}, offers dimension reduction at a low computational cost. In this scheme, $\mb X$ is mapped to $\mb X \mb A$, where $\mb A$ is a $p \times d$ matrix typically with i.i.d.\ random entries. Efficient implementations are discussed in \citet{achlioptas2001database, li2006very} and some numerical results on random projections and a wider literature review are in \citet{fradkin2003experiments, Vempala2005}. The software package \emph{Vowpal Wabbit} \citep{langford2007vowpal} is a popular learning system for large-scale datasets that uses sparse random projections.

A separate line of work has considered pre-multiplying $\mb X$ with a random matrix $\mb A \in \R^{m \times n}$ to produce a reduced matrix $\mb A\mb X \in \R^{m \times p}$, known as a \emph{sketch}. Though the dimension $p$ is not reduced, when $n$ is large, performing OLS on the sketched matrix may be possible despite the computational infeasibility of applying least squares directly to $\mb X$. A number of works have studied properties sketched least squares (see \citet{boutsidis2009random, drineas2011faster, Mahoney2011, Pilanci2015} and references therein) whilst \citet{Pilanci2014} propose an iterative variant of this scheme. \citet{Yang2015}  considers sketching ideas in the context of kernel ridge regression.


\section{$b$-bit min-wise hashing} \label{sec:MRS}
Given a sparse design matrix $\mb{X} \in \R^{n \times p}$, the aim of dimension reduction is to map this to a  compressed matrix $\mb{S} \in \R^{n \times d}$, in a way that is computationally efficient and such that the relevant information in $\mb{X}$ is preserved in $\mb{S}$.
Section~\ref{sec:constr_S} describes the mapping to $\mb{S}$ under $b$-bit min-wise hashing for binary data, as proposed in \citet{li2011theory} and \citet{li2011hashing}. The construction may seem unintuitive at first sight, but we will try to shed light on why the scheme works for linear and interaction models throughout the manuscript.

\subsection{Notation}
Given a matrix $\mb{U}$, we will write $\mb{u}_i$ and $\mb{U}_j$ for the $i$th row and $j$th column respectively, where both are to be regarded as column vectors. The $ij$th entry will be denoted $U_{ij}$.
A vector of $1$'s will be denoted $\mb 1$.


When the parentheses following probability and expectation signs, $\pr$ and $\E$, enclose multiple potential sources of randomness, we will sometimes add subscripts to indicate what is being considered as random. For example, if $U$ and $V$ are random variables, we may write $\E_U (U |V)$ for the conditional expectation of $U$ given $V$, and $\E_{U, V}(U + V)$ for the expected value of $U + V$.

\subsection{Construction of $\mb{S}$ with $b$-bit min-wise hashing and binary variables} \label{sec:constr_S}
The compressed matrix $\mb S$ generated by $b$-bit min-wise hashing consists of blocks of size $2^b$, where we may choose the number of blocks $L$. Each block is created using a random permutation and the blocks of columns form a collection of $L$ i.i.d.\ random matrices. 

%
There are three steps to the construction.
\begin{itemize}
 \item[\Step{1}] Generate a random permutation of the set $\{1, \ldots, p\}$, $\pi_l$, and permute the columns of $\mb{X}$ according to this permutation.
 \item[\Step{2}] Search along each row of the permuted design matrix (in order of increasing column index) and record in the vector $\mb{H}_l \in \mathbb{N}^n$ the indices of the variables (indexed as in the original order) with the first non-zero value or the vector $\mb{M}_l \in \mathbb{N}^n$ the indices of the variables (indexed as in the permuted order) with the first non-zero value.
 \item[\Step{3}] Form $\mb{S}_l \in \{0,1\}^{n\times 2^b}$ with $i$th row given by the last $b$ bits of the binary representation of the $i$th entry of $\mb{M}_l$.  For example, when $b=1$, all odd numbers in $\mb{M}_l$ map to the vector $(0,1)$, whereas all even numbers map to $(1,0)$.
\end{itemize}
This construction is illustrated for a toy example in Table~\ref{tab:illus1}.
{\edef\savedbaselineskip{\the\baselineskip\relax}
\begin{table}[ht]
\centering
\begin{tabular}{|p{8.1cm}| c| c|}
\hline
$ {\baselineskip=\savedbaselineskip \mb{X} = \bordermatrix{& 1 & 2 & 3 & 4 \cr
                & \cdot & 1 & \cdot & 1 \cr
                & \cdot & \cdot & 1 & 1\cr
                & 1 & \cdot & 1 & \cdot \cr
                & \cdot & 1 & 1 & \cdot \cr
                & 1 & 1 & \cdot & \cdot} } \stackrel{\pi_l = 2314}{\mapsto}                
                {\baselineskip=\savedbaselineskip \bordermatrix{& 3 & 1 & 2 & 4 \cr
                & \cdot & \cdot & \bf 1  & 1 \cr
                & \bf 1 & \cdot & \cdot & 1 \cr
                & \bf 1 & 1 & \cdot & \cdot \cr
                & \bf 1 & \cdot & 1 & \cdot \cr
                & \cdot & \bf 1 & 1 & \cdot}}$
 & 
 $\mb{H}_l = \begin{pmatrix} 2 \\ 3 \\ 3 \\ 3 \\ 1 \end{pmatrix}$,
$\mb{M}_l = \begin{pmatrix} 3 \\ 1 \\ 1 \\ 1 \\ 2 \end{pmatrix}$ 
 & 
 $\mb{S}_l = \begin{pmatrix}
 0 & 1 & 0 & 0 \\
 0 & 0 & 0 & 1 \\
 0 & 0 & 0 & 1 \\
 0 & 0 & 0 & 1 \\
 0 & 0 & 1 & 0 \end{pmatrix}$
 \\
 \emph{Step 1}: non-zero indices whose variable indices will appear in $\mb{H}_l$ in Step 2 are in bold. & \emph{Step 2}. & \emph{Step 3}. \\
 \hline 
\end{tabular}
\caption{Steps 1--3 applied to a toy example with $b=2$. Dots represent zeroes.}
\label{tab:illus1}
\end{table}}

We can think of each column of $\mb S_l$ as representing different categories for the observations. The matrix $\mb S_l$ itself codes for the assignment of the different rows of $\mb X$ to the different categories. Different blocks $\mb S_l$ then represent different random categorisations. Identical rows will always be assigned the same categories and the more different the rows are, the less likely they are to be assigned the same category. The notion of difference here is that of \emph{resemblance}; see Section~\ref{sec:Jaccard}

Note that one would not necessarily follow the above steps when implementing $b$-bit min-wise hashing. In practice, one would not store the entire matrix of signs nor all the random permutations.
 In an implementation, hash functions \citep{carter1979universal} would be used to create the matrix $\mb S$ deterministically, though it is 
beyond the scope of this paper to go into the details; see \citet{li2013b} for more information and further computational improvements. With this approach, $\mb S$ would be 
created row-by-row, and only a single observation from $\mb X$ would need to be kept in memory at any one time. Furthermore, many rows could be created in parallel. Other ideas such as one-permutation hashing \citep{li2012one} can also be used to speed up the pre-processing step.

  
\subsection{Continuous data and additional randomisation} \label{sec:shuffled}

For continuous data, we introduce a modification where we replace the map extracting the last $b$ bits by $L$ random maps in the following way.
Fix $b$ and let $\mbb\Psi \in \{1, \ldots, 2^b\}^{p \times L}$ be a random matrix with independent entries each having the uniform distribution on the set $\{1,\ldots,2^b\}$. We then create $\mb S$ by modifying the previous Step 3 to the following.

\begin{itemize}
 \item[\Step{3}]  Form $\mb S_l \in \{0, 1\}^{n \times 2^b}$ with $i$th row all zero except component $\Psi_{H_{il} l}$ takes the value 1.
 \item[\Step{4}]  If $\mb X$ is not binary, multiply the $i$th row of $\mb{S}_l$ by $X_{i{H_{il}}}$.
\end{itemize}

This generalisation is motivated by our theoretical results on how well the column space of $\mb S$ can capture different sorts of signals (see Section~\ref{sec:un-scaled}).

Let $\mb z_i= \{k:X_{ik} \neq 0\}$ be the set of variable indices whose entries have non-zero values for the $i$th observation.
Performing the steps above for all $l=1,\ldots,L$, we get $n \times L$ matrices $\mb{H}$, and $\mb{M}$ given by
\begin{align}
 H_{il} = &  \argmin{k\in \mb z_i} \pi_l(k), \label{eq:def_H}\\
 M_{il} = &  \min_{k\in \mb z_i} \pi_l(k) = \pi_l(H_{il}), \label{eq:def_M}
\end{align}
The matrix $\mb{S}$ is a binary $n\times 2^bL$ matrix.  With a slight abuse of notation, we will
denote by $\mb S_{ilc}$ the $c$th entry in the $l$th block of $\mb S$:
\begin{align}
  S_{ilc} := S_{i (c+(l-1)2^b)} = X_{iH_{il}}\ind_{\{\Psi_{H_{il}l}=c\}}, \qquad \mbox{for } c=1,\ldots,2^b. \label{eq:def_S}
\end{align}

If not stated otherwise, we will work with this second randomised
variation of $b$-bit min-wise hashing from now on. We emphasise that we do not make the claim this version is to be preferred over the original proposal of \citet{li2011theory} and \citet{li2011hashing} when data is binary. We simply introduce the additional randomisation here to simplify the analysis. We note that the two versions are 
essentially identical for all practical purposes when $b$ is not too large.

\subsection{The resemblance kernel} \label{sec:Jaccard}
We now briefly describe the connection between $b$-bit min-wise hashing and the resemblance kernel alluded to earlier. This is not needed for the rest of the paper, though it provides some intuition for the scheme. A more detailed analysis from this perspective is carried out by \citet{li2011hashing} and we refer the reader to \citet{Hofmann2008} for a review of kernel methods and the kernel trick.

Suppose $\mb X$ is binary.
Consider the normalised Gram matrix of the compressed design $\mb S$ from (randomised) $b$-bit min-wise hashing, $\mb S \mb S^T /L$. The expected value of the $ij$th component may be calculated as follows.
\begin{align*}
 \E_{\mbb\pi, \mbb\Psi}(\mb s_i^T \mb s_j/L) &= \frac{1}{L}\sum_{l=1}^L \sum_{c=1}^{2^b} \E_{\mbb\pi, \mbb\Psi} (\ind_{\{\Psi_{H_{il}l}=c\}}\ind_{\{\Psi_{H_{jl}l}=c\}}) \\
 &= \pr(\Psi_{H_{il}l} = \Psi_{H_{jl}l}) \\
 &= \pr(\Psi_{H_{il}l} = \Psi_{H_{jl}l} | H_{il}=H_{jl})\pr(H_{il}=H_{jl}) \\
 &\qquad +\pr(\Psi_{H_{il}l} = \Psi_{H_{jl}l} | H_{il}\neq H_{jl})\{1- \pr(H_{il}=H_{jl})\} \\
 &= \frac{|\mb z_i \cap \mb z_j|}{|\mb z_i \cup \mb z_j|}( 1-2^{-b}) +2^{-b}.
\end{align*}
Thus the $ij$th entry is an average of $L$ i.i.d.\ random variables with expectation a constant plus a constant times the resemblance between the $i$th and $j$th rows of $\mb X$. If an intercept term is included when regressing on $\mb S$, the additive constant plays no part, and the scaling would be absorbed into the scaling of the regression coefficients.
We also note that when $\mb X$ is continuous, the resulting kernel is similar to the the CoRE kernels of \citet{Li2014}.

Now as the resemblance kernel is positive definite,
the theory surrounding the kernel trick tells us that 
any $\ell_2$-regularised regression on $\mb S$ is effectively approximating a regularised regression on transformed data $\phi(\mb x_i)$ where $\phi : \{0,1\}^p \to \mathcal{H}$ and $\mathcal{H}$ is a high-dimensional inner product space (the feature space). This space may be taken to be a reproducing kernel Hilbert space (RKHS), and then $\phi$ and $\mathcal{H}$ are uniquely defined.

Although this is encouraging, the kernel trick does not guarantee that regression on $\mb S$ will necessarily have good predictive properties for models of interest. To gain a better understanding, we must study the regularisation properties of the resemblance kernel itself: what characterises those elements of the associated RKHS $\mathcal{H}$ that have low norm and thus will be penalised less?

A direct analysis of the RKHS corresponding to the resemblance kernel in those terms seems challenging. We take a different approach and explicitly construct regression coefficients for $\mb S$ that approximate signals of interest. By showing that particular signals can be approximated well, we are indirectly discovering elements of $\mathcal{H}$ with low RKHS norm (see also Section~\ref{sec:RKHS_interpret} for more details).

\section{Approximation error}
\label{section:main}
In this section, we present results that bound the expected prediction error when performing regression on the reduced design matrix $\mb{S}$ in the contexts of the linear and logistic regression models. Note that throughout the rest of the manuscript, by $b$-bit min-wise hashing we are referring to the randomised variant described in Section~\ref{sec:shuffled}.
Let $q_i$ be the number of non-zero entries in the $i$th row of $\mb X$, and let $\delta_i=q_i/p$ be the row sparsity.
We will assume that the signal we wish to approximate for the $i$th observation takes the form
\begin{align} \label{eq:row_scale}
\kappa(\delta_i) \mb{x}_i^T \mbb{\beta}^*.
\end{align}
Here $\mbb{\beta}^* \in \R^p$ is an unknown vector of coefficients and the function $\kappa$ allows the $i$th linear predictor to be scaled in a way which depends on the number of non-zero entries in the $i$th row of $\mb X$.
Some normalisations of special interest include:
\begin{enumerate}[(a)]
\item $\kappa(\delta)$ constant. This yields standard linear or logistic regression models.
\item $\kappa(\delta)\propto \delta^{-1/2}$. In text analysis with a bag of words representation of documents, rows of $\mb{X}$ are often scaled to have the same $\ell_2$-norm to help balance situations when documents vary greatly in length \citep{row_normalise}. When $\mb X$ is binary, this is exactly achieved by  taking $\kappa(\delta) = p^{-1/2}\delta^{-1/2}$, so $\kappa(\delta_i)=q_i^{-1/2}$.
\item $\kappa(\delta)\propto \delta^{-1}$. This leads to a $\ell_1$-norm scaling as opposed to the $\ell_2$-norm scaling mentioned above. 
\end{enumerate}

Throughout we will assume that $\mb X \in [-1,1]^{n \times p}$, so the entries in $\mb X$ are bounded. This covers the important case of binary design but also allows for real-valued entries.
 
The first step in obtaining our prediction error results is to construct a vector $\mb b^*$ such that $\mb s_i^T \mb b^*$ is close to $\kappa(\delta_i)\mb x_i^T \mbb \beta^*$ on average.

\subsection{Un-scaled signals} \label{sec:un-scaled}
We will first consider un-scaled signals where $\kappa(\delta)$ in~\eqref{eq:row_scale} is a constant. Non-constant row-scaling is treated in more detail in the Section~\ref{sec:row_scaling}.
To begin with we will assume that $q_i=q \geq 1$ for all $i=1,\ldots,n$, a restriction which simplifies the results but highlights some interesting properties of $b$-bit min-wise hashing. Unequal row sparsity is treated in detail in the appendix in Section~\ref{sec:unequal} but a sketch of the results are given just below Theorem~\ref{thm:approx_error_main_binary}.

To simplify notation, we first introduce the following norm for $\mbb \beta\in \mathbb{R}^p$,
\begin{equation}\label{eq:norm_b} \norms{\mbb \beta}_b^2 :=  \norms{\mbb \beta}_2 ^2  + (2^b-2) \sum_{k=1} ^p \frac{\norms{\mb X_k}_2 ^2}{n} \beta_k^2 .\end{equation}
For $b=1$, we have of course that $\norms{\mbb \beta}_b^2= 2\norms{\mbb\beta}_2^2$. For larger values of $b$, the norm is influenced more heavily by the second term which can be seen to be the weighted version of the $\ell_2$-norm, where the weight of each variable is proportional to its squared $\ell_2$-norm. We will first discuss how well the original signal can be approximated with the column space of the matrix $\mb S$ generated by the $b$-bit min-wise hashing operation.

\begin{thm} \label{thm:approx_error_main_binary}
 Let $\mb S$ be the matrix generated by $b$-bit min-wise hashing.
Then there exists a vector $\mb{b}^{*} \in \R^{2^bL}$ with the following properties.
\begin{enumerate}[(i)]
 \item The approximation is unbiased: $\E_{\mbb{\pi}, \mbb{\Psi}}(\mb{S}\mb{b}^*) = \mb{X}\mbb{\beta}^*$.
  \item The norm is bounded by
 \begin{equation*}
 \E_{\mbb\pi, \mbb\Psi}(\|\mb b^*\|_2^2) \leq \frac{(2-\delta)q}{L(1-2^{-b})}\|\mbb\beta^*\|_2^2.
 \end{equation*}
 \item  The approximation error is bounded by
 \begin{equation*}
   \frac{1}{n} \E_{\mbb{\pi}, \mbb{\Psi}}(\norms{\mb{S}\mb{b}^* - \mb{X}\mbb{\beta}^*}_2 ^2) \leq \frac{(2-\delta)q}{2^bL(1-2^{-b})} \norms{\mbb \beta^*}_b ^2. 
  \end{equation*}
Specifically, for $b=1$, $\E_{\mbb{\pi}, \mbb{\Psi}}(\norms{\mb{S}\mb{b}^{*} - \mb{X}\mbb{\beta}^*}_2 ^2)/n  \leq 
 (2-\delta)q \norms{\mbb \beta^*}_2 ^2 / L.$
\end{enumerate}
\end{thm}
A form of the approximation error (iii) and the norm bound (ii) continue to be valid in the non-equal sparsity case under a mild restriction on the size of $L$, where we get instead of (iii) the bound
\begin{equation*}
   \frac{1}{n} \E_{\mbb{\pi}, \mbb{\Psi}}(\norms{\mb{S}\mb{b}^* - \mb{X}\mbb{\beta}^*}_2 ^2) \leq \frac{6\bar{q}}{2^bL(1-2^{-b})} \norms{\mbb \beta^*}_b ^2,
  \end{equation*}
where $\bar{q}$ is the the average of the $q_i$;
see Theorem~\ref{thm:unequal_sparsity} in the appendix for details.

The results above show that the signal $\mb X \mbb{\beta}^*$ can be well approximated by a linear combination of the columns in the matrix $\mb S$ if we generate a sufficiently large number of permutations $L$, especially for sparse data matrices. Another useful property of $\mb{b}^*$ here, aside from the approximation accuracy it delivers, is given in (ii): on average, $\norms{\mb{b}^*}_2 ^2$ is small when $L$ is large. This proves to be useful when studying the application of ridge regression.

This result has interesting implications for the resemblance kernel and its RKHS $\mathcal{H}$. In particular, it shows that if we constrain the input space to contain those vectors with sparsity $q$, linear functions $f_{\mbb\beta}$ defined by coefficients $\mbb\beta \in \R^p$ with $\sum_j \beta_j=0$ have RKHS norm satisfying $\|f_{\mbb\beta}\|_{\mathcal{H}}^2 \leq (2-\delta)q \|\mbb\beta\|_2^2$. As these properties of the RKHS are not directly used in any subsequent results, we defer formal presentation of these facts to Section~\ref{sec:RKHS_interpret} in the appendix.

Whilst the bound on the expectation of $\|\mb b^*\|_2^2$ is almost constant as $b$ changes, the approximation error bound (iii) does vary with $b$. Consider the case where $\mb X$ is binary and let $\gamma_k = \|\mb X_k\|_2^2/n$ be the column sparsity. Typically one would expect $\|\mbb\beta^*\|_2^2$ to be significantly larger than $\sum_{k=1}^p \gamma_k{\beta_k^*}^2$ and thus increasing $b$ by 1 almost halves the approximation error when $b$ is small.

A proof of Theorem~\ref{thm:approx_error_main_binary} is given in Section~\ref{app:approx} of the appendix; here we briefly sketch some of the main ideas. Note that
\begin{equation} \label{eq:sketch1}
\E_{\mbb\pi, \mbb\Psi}( \mb S\mb b^*) = \sum_{l=1}^L \E_{\mbb\pi, \mbb\Psi}\bigg(\sum_{c=1}^{2^b} \mb S_{lc}b^*_{lc}\bigg).
\end{equation}
We construct $\mb b^*$ with the following two properties: each of the $L$ blocks of $\mb b^*$ are i.i.d.\ with the $l$th block only depending on $\pi_l$ and $\mbb\Psi_l$; and each of the $L$ summands in \eqref{eq:sketch1} equals $\mb X\mbb\beta^* /L$. With each of the $L$ summands being unbiased in this way, we see that the approximation error is controlled by the variance of the sum; this variance scales as $1/L$ since the summands are i.i.d.

At first sight it may seem surprising that it is possible to exhibit a $\mb b^*$ with each block having the unbiasedness property discussed above. However, the following construction gives an indication of the possibilities. 
Using our convention that the $c$th component of the $l$th block of $\mb b^*$ is indexed as $b^*_{lc}:= b^*_{c+(l-1)2^b}$, consider taking
\begin{equation} \label{eq:sketch2}
b^*_{lc} =\frac{q}{L} \sum_{k=1}^p\beta^*_k \frac{\ind_{\{\Psi_{lk}=c\}}-2^{-b}}{1-2^{-b}}.
\end{equation}
Then writing $\mbb\psi = \mbb\Psi_1$, $\pi=\pi_1$, $H_i=H_{i1}$ we have
\begin{align}
\frac{L}{q}\E_{\pi,\mbb\psi} \bigg(\sum_{c=1}^{2^b} S_{lc}b^*_{1c}\bigg) &= \E_{\pi,\mbb\psi} \bigg(\sum_{c=1}^{2^b} \sum_{j=1}^p X_{ij}\ind_{\{H_i=j, \psi_j=c\}} \sum_{k=1}^p\beta^*_k \frac{\ind_{\{\psi_{k}=c\}}-2^{-b}}{1-2^{-b}}\bigg) \notag\\
&= \E_{\pi,\mbb\psi} \bigg( \sum_{j=1}^p X_{ij}\ind_{\{H_i=j\}} \sum_{k=1}^p\beta^*_k \frac{\ind_{\{\psi_{k}=\psi_j\}}-2^{-b}}{1-2^{-b}}\bigg) \label{eq:sketch3}.
\end{align}
Now since $\E_{\mbb\psi}\{(\ind_{\{\psi_{k}=\psi_j\}} -2^{-b})/(1-2^{-b})\}=\ind_{\{k=j\}}$ we see the above display equals
\begin{align*}
q \sum_{k=1}^p X_{ik}\beta^*_k \pr_{\pi}(H_i=k) = \mb X\mbb\beta^*.
\end{align*}
The final line uses the fact that for $k$ with $X_{ik} \neq 0$, $\pr_{\pi}(H_i=k)$ is the reciprocal of the number of non-zero entries in the $i$th row of $\mb X$; with our simplifying assumption of equal row sparsity, this is precisely $1/q$. Note one could scale the rows of $\mb S$ according to the number of non-zeroes in each row to achieve unbiasedness in the case of unequal row sparsity. However as shown in Section~\ref{sec:unequal}, it turns out that by incurring some bias one can still keep the approximation error low even in this situation without having to perform any sort of scaling. 

The form of $\mb b^*$ used in the proof of Theorem~\ref{thm:approx_error_main_binary} differs slightly from that in \eqref{eq:sketch2} by introducing a random weight multiplying each coefficient  that decays as $\pi_l(k)$ increases. This reduces the variance and yields the approximation error in (iii) that has a factor $q$ rather than the factor of $p$ which would be obtained from \eqref{eq:sketch2}.

\subsection{Row-scaled signals} \label{sec:row_scaling}
We now turn to the more general setting with unequal row sparsity and signal given by \eqref{eq:row_scale}.
We consider the family of scaling functions
 $\delta \mapsto (\delta_{\min}/\delta)^a$ where $\delta_{\min}=\min_i \delta_i$, for $1/2 \leq a \leq 1$. Including $\delta_{\min}$ in the scaling functions means that were the row sparsity to be equal, the approximation error here would be of the same form as that considered in Theorems~\ref{thm:approx_error_main_binary}. We could alternatively replace $\delta_{\min}$ with the average of the $\delta_i$ for the same effect, but using $\delta_{\min}$ helps to simplify the results. Writing $q_{\min} = \min_i q_i$, we have the following results.

\begin{thm}
\label{thm:scaling_adapt}
Let $L\geq 5$ and assume $\delta_{\min} \leq 1/2$  if $a=1/2$, and $L>2/(2a-1)$  if $a>1/2$. Then there exists $\mb b^* \in \R^L$ depending on $a$ such that the approximation error satisfies
\begin{align*}
& \frac{1}{n}\sum_{i=1}^n \E_{\mbb\pi, \mbb\Psi}[\{ (\delta_{\min}/\delta_i)^a\mb x_i^T\mbb\beta^* - \mb s_i^T \mb b^*\}^2] \leq \\ 
&\qquad \begin{cases}
\dfrac{q_{\min}}{2^b L (1-2^{-b})}\|\mbb\beta^*\|_b^2\log\{4\log(L)/\delta_{\min}\} \quad & \text{if } \; a=1/2,\\
&\\
\dfrac{q_{\min}}{2^b L (1-2^{-b})} \|\mbb\beta^*\|_b^2 \dfrac{1}{2a-1} [\log\{2(2a-1)L\}]^{2a-1}
 \quad & \text{if }\; 1/2 <a\leq 1,
\end{cases}
\end{align*}
and the norm of $\mb b^*$ is bounded in expectation by 
\[
\E_{\mbb\pi, \mbb\Psi}(\|\mb b ^*\|_2^2) \leq \begin{cases}
\dfrac{q_{\min}\log\{4\log(L)/\delta_{\min}\}}{L(1-2^{-b})}\|\mbb\beta^*\|_2^2 \qquad & \text{if } \; a=1/2,\\
&\\
\dfrac{1}{2a-1}\dfrac{q_{\min}[\log\{2(2a-1)L\}]^{2a-1}}{L(1-2^{-b})} \|\mbb\beta^*\|_2^2
 \qquad & \text{if }\; 1/2 <a\leq 1.
\end{cases}
\]
\end{thm}
The min-wise hashing based dimension reduction scheme appears to be well-suited to approximating signals scaled by a power of the sparsity, with the approximation error only incurring a further multiplicative term involving $\log(L)$ compared to the results of Theorem~\ref{thm:approx_error_main_binary}.

We now briefly outline how we construct coefficient vectors $\mb b^*$ achieving the bounds above. Consider the following refinement of \eqref{eq:sketch2}:
\begin{equation*}
b^*_{lc}= \frac{1}{L} \sum_{k=1}^p\beta_k^* \frac{\ind_{\{\Psi_{lk}=c\}} -2^{-b}}{1-2^{-b}} w_{\pi_l(k)},
\end{equation*}
where $\mb w \in \R^p$ is a vector of non-negative weights. Arguing as in \eqref{eq:sketch3} but replacing $q\beta_k^*$ with $\beta_k^*w_{\pi(k)}$ we arrive at
\[
L\E_{\pi,\mbb\psi} \bigg(\sum_{c=1}^{2^b} S_{lc}b^*_{1c}\bigg) = \sum_{k=1}^p X_{ik}\beta^*_k \E_\pi (\ind_{\{H_i=k\}}w_{\pi(k)}).
\]
Recall that writing $M_i=M_{i1}$, $M_i = \pi(H_i)$, the position of the first non-zero entry in row $i$ under permutation $\pi$. Note that $H_i$ and $M_i$ are independent. Now for large $p$, $M_i$ behaves roughly like a geometric random variable with parameter $\delta_i$. Thus for $k$ with $X_{ik} \neq 0$,
\[
\E_\pi (\ind_{\{H_i=k\}}w_{\pi(k)}) = \E_\pi(\ind_{\{H_i=k\}} w_{M_i}) \approx \frac{1}{p\delta_i}\sum_{\ell=1}^p w_\ell \delta_i (1-\delta_i)^{\ell-1} = \frac{1}{p}\sum_{\ell=1}^p w_\ell (1-\delta_i)^{\ell-1}.
\]
If $w_{\ell+1} = p(-1)^\ell\kappa^{(\ell)}(1)/\ell!$ we see that the RHS resembles a Taylor series of $\kappa(\delta_i)$ about 1. In this way we can approximate a large family of row-scaled signals.

\section{Prediction error}
The approximation error results in the three previous sections allow us to derive bounds on the prediction errors for linear and logistic regression models with potentially row-scaled data. Here we will present results under the assumption of $q$ non-zero entries per row and also where the scaling function $\kappa$ is proportional to the square-root function 
\begin{equation} \label{eq:kappa_0}
\kappa_0(\delta)=\sqrt{\delta_{\min}/\delta}.
\end{equation}
However, all of the approximation error results can be extended to results on prediction error via general theorems on prediction error we present in Section~\ref{sec:pred_res}. In particular, Theorem~\ref{thm:unequal_sparsity} can be used to show that versions of the equal row sparsity results hold more generally with $q$ replaced by the average number of non-zeroes per row $\bar{q}$ provided $L$ is not excessively large.

\subsection{Linear regression models} \label{sec:OLS_main}
Assume we have the following approximately linear model:
\begin{equation} \label{eq:linmod}
 Y_i = \alpha^* + \kappa(\delta_i) \mb x_i^T\mbb{\beta}^* + \varepsilon_i, \qquad 1=1,\ldots,n.
\end{equation}
Here $\alpha^*$ is the intercept and $\mb x_i \in [-1,1]^p$. We assume that the random noise $\mbb{\varepsilon} \in \R^n$ satisfies $\E (\varepsilon_i) = 0$, 
$\E(\varepsilon_i ^2) = \sigma^2$ and $\Cov(\varepsilon_i, \varepsilon_j)=0$. 

Our results here give bounds on a mean-squared prediction error (MSPE) of the form
\begin{equation} \label{eq:MSPE}
\MSPE((\hat{\alpha}, \hat{\mb b})) := \frac{1}{n}\sum_{i=1}^n \E_{\mbb \varepsilon, \mbb \pi,\mbb \Psi} \{(\alpha^* + \kappa(\delta_i) \mb x_i^T\mbb{\beta}^* - \hat{\alpha} - \mb S \hat{\mb b})^2\}
\end{equation}
where $\hat{\alpha}$ and $\hat{\mb b}$ are the estimated intercept and regression coefficients arising from regression on $\mb S$.
Note we consider a denoising-type error: the error on the data used to fit the regression coefficients. 
 Bounds on the prediction error at new observations would require conditions on the distribution of observations and we have avoided making any such assumptions for the results here. 

\subsubsection{Ordinary least squares} \label{sec:main_OLS_MSPE} 
Perhaps the simplest way to estimate the linear model is to apply a least squares estimator,
\begin{equation} \label{eq:est_ols}
(\hat{\alpha}, \hat{\mb b}) := \argmin{(\alpha, \mb b) \in \R \times \mathbb{R}^{2^bL}} \|\mb Y - \alpha \mb1 - \mb S \mb b\|_2^2 ,
\end{equation}
to the matrix $\mb S$.
We have the following theorem.

 


\begin{thm} \label{thm:main_OLS_pred_error}
Let $(\hat{\alpha}, \hat{\mb b})$ be the least squares estimator (\ref{eq:est_ols}). We have the bound
 \begin{equation*}
  \MSPE ((\hat{\alpha}, \hat{\mb{b}})) \;\leq \;  \frac{C }{2^bL(1-2^{-b})} \; \norms{\mbb \beta^*}_b ^2   + 2^b L \frac{\sigma^2}{n} .
 \end{equation*}
For equal row sparsity $\delta$ we have $ C=(2-\delta)q.$  For unequal row sparsity, when $\kappa=\kappa_0$ as in  \eqref{eq:kappa_0}, the result holds for 
$ C=q_{\min}\log\{ 4\log(L)/\delta_{\min} \} .$
\end{thm}

An optimal choice $L_b^*$ of $L$ will balance the approximation error and variance contributions (first and second term on the right hand side respectively). In the equal row sparsity we arrive at
\[
L_b^* = \frac{\sqrt{(2-\delta)q n}}{2^b\sqrt{1-2^{-b}}} \|\mbb\beta^*\|_b 
\]
which yields an optimal MSPE of the order $\sigma \sqrt{q/n} \|\mbb\beta^*\|_b$. 
If we ignore log terms the rate is  analogous in the case of uneven row-sparsity. The slow rate in $n$ seems unavoidable if we do not make stronger conditions on the design. Indeed, a similar error rate is obtained in Theorem 21 of \citet{maillardmunos2012} and in \citet{kaban2014new} for OLS following dimension reduction by random projections. More precisely: projecting $K$ times with a random projection, followed by an OLS estimation is shown in \citet{kaban2014new} to lead to a bound on $\MSPE$ of
\begin{equation} \label{eq:Kaban}
 \frac{1 }{K} \; \norms{\mbb \beta^*}_\kappa^2   + K \frac{\sigma^2}{n} ,
 \end{equation}
where the norm $\norms{\cdot}_\kappa$ depends on the eigenvalue structure of the design matrix. In contrast the bound we have above for min-wise hashing depends in contrast on the sparsity $q$ through the constant $C$. The bound \eqref{eq:Kaban} is otherwise structurally identical to the bound for $b$-bit min-wise hashing above, and the role of the number $L$ of projections is now taken by the number $K$ of random projections. The optimal values of $K$ and $L$ are both of order $\sqrt{n}$, leading to the same convergence rate of the risk as $n\rightarrow\infty$.

To better understand the implications of Theorem~\ref{thm:main_OLS_pred_error}, it is helpful to fix the size of the signal so that $\|\mb X\mbb \beta^*\|_2^2/n=1$, and look at whether we can show consistency of the method as both $p,n\rightarrow\infty$.
If the signal is spread out and all variables have the same sparsity, $\|\mbb \beta^*\|_b$ will be of order $\sqrt{p/q}$ and the MSPE will vanish when $p/n\rightarrow 0$, which excludes the high-dimensional setting. 

However, now assume that the signal is concentrated on a fixed set of  variables. The norm $\|\mbb \beta^*\|_b$ is then constant as $p$ increases and all that is required for consistency is $q/n\rightarrow 0$ (or $q_{\min}/n\rightarrow 0$ for the more general case of uneven row-sparsity). 
 
An interesting scenario is one of increasing variable sparseness. In many applications, the more predictor variables are added the sparser they tend to become. In text analysis, the first block of predictor variables might encode the presence of individual words. The next block might code for bigrams and the following, higher order $N$-grams. With this design, predictor variables in each successive block become sparser than the previous. It is then interesting to consider how much the MSPE can increase if we add a block with many sparse variables which contain no additional signal contribution. The result above indicates that the MSPE only increases as $\sqrt{q}$. Adding a block of several million (sparse) bigrams might thus have the same statistical effect as adding several thousand (denser) unigrams (individual words).

We now comment the optimal choice of $L$ and computational complexity. If we assume fixed $\|\mbb\beta^*\|_2$ and $n = O(q)$, which is all that would be required to keep the prediction error bounded asymptotically, then the optimal dimension of the min-wise projection scales as  $L_b^* = O(q)$, considering $b$ fixed here. This dimension  will in general be a substantial reduction over the original dimension of the data, $p$, and would result in a correspondingly large reduction in the computational cost of regression. Indeed, ridge regression or the LAR algorithm \citep{efron04least} applied to $\mb X$ would have complexity $O(q^2 p)$, and one would expect that the Lasso \citep{tibshirani96regression} would have similar computational cost. In contrast, OLS applied to $\mb S$ would only require $O(q^3)$ operations, an improvement of $q/p$.
The discussion above considered an optimal choice of $L \approx L_b^*$.  Even if we cannot afford to work with the  optimal dimension $L_b^*$ for computational reasons, the bound will still be useful for smaller values of $L$. The guarantee on prediction accuracy could not be obtained if, for example, simply a random subset of $L$ predictors were chosen and the remaining ones discarded.

The dependence of the bound on $b$ is also interesting: a minimum value occurs for $b=1$. However, this would imply a larger value of $L^*_b$. Note the memory requirement for storing $\mb S$ would be $O(nL^*_b b)$ as $b$ bits would be required to store the locations of each of the $nL^*_b$ nonzeroes. We see that with a constraint on $nbL$ or on the number of permutations $L$, larger values of $b$ are more favourable, particularly with high sparsity, as this would tend to make $\|\mbb\beta^*\|_b$ not much larger than $\|\mbb\beta^*\|_2$. A different perspective on the optimal choice of $b$ based on the variance of inner products of rows of $\mb S$ is taken in \citet{li2011theory}, with similar conclusions.

\subsubsection{Ridge regression}

Instead of using a least-squares estimator on the transformed data matrix $\mb S$ we can also apply ridge regression \citep{hoerl1970ridge}. 
For a given $\lambda > 0$, the regression coefficients are found by
\begin{equation}
\label{eq:main_ridge_def}
  (\hat{\alpha}_\lambda, \hat{\mb{b}}_\lambda) := \argmin{(\alpha, \mb{b}) \in \R \times \R^L} \norms{\mb{Y} - \hat{\alpha} \mb{1} - \mb{S}\mb{b}}_2 ^2 \; \text{ such that } \; \norms{\mb{b}}_2 ^2 \leq \lambda,
\end{equation}
The theorem below gives a bound on the MSPE of $(\hat{\alpha}, \hat{\mb{b}}_\lambda)$.

\begin{thm} \label{thm:main_ridge_pred_error}

There exist regularisation parameters $\lambda$ depending on $\mbb\beta^*$ and $\mb S$ such that 
 \begin{equation*}
  \MSPE ((\hat{\alpha}_\lambda, \hat{\mb{b}}_\lambda)) \;\leq \; \sigma \sqrt{\frac{ 2 C }{(1-2^{-b})n}}\|\mbb\beta^*\|_2 +  \frac{C}{2^bL (1-2^{-b})} \norms{\mbb \beta^*}_b ^2 +\frac{\sigma^2}{n}.
 \end{equation*}
Here the value of  $C$ is defined as in Theorem~\ref{thm:main_OLS_pred_error} by
$ C=(2-\delta)q $ for equal row sparsity $\delta$  and  $ C = q_{\min}\log\{4\log(L)/\delta_{\min}\} $ for $\kappa=\kappa_0$ and unequal row-sparsity.
\end{thm}
The ridge regression result for large $L$ is similar to that for OLS with an optimal $L^*_b$, though there is a small difference: the leading terms are $\sigma \|\mbb\beta^*\|_2 \sqrt{q/n}$ and $\sigma \|\mbb\beta^*\|_b \sqrt{q/n}$ respectively. Ridge regression takes advantage of the fact that not only do we have a $\mb b^*$ such that $\mb S \mb b^*$ and $\mb X\mbb\beta^*$ are close, we also know that there is a $\mb b^*$ with this property that has low $\ell_2$-norm. Our bound on the expected squared $\ell_2$-norm of $\mb b^*$ ((ii) in Theorem~\ref{thm:approx_error_main_binary}) does not depend much on $b$.
In contrast, OLS only makes use of the approximation error result, (iii) in Theorem~\ref{thm:approx_error_main_binary}.

Note that when $L$ is large, regardless of the value of $b$, ridge regression on $\mb S$ approximates a kernel ridge regression using the resemblance kernel (see Section~\ref{sec:Jaccard}). The MSPE of a kernel ridge regression with the resemblance kernel should of course not depend on $b$, and this observation largely agrees with our result.

Another key difference between ridge regression and OLS here is the following: achieving a good prediction error with OLS hinges on a careful choice of $L$. In contrast, with ridge regression, $L$ can (and should) be chosen very large, from a purely statistical point of view. However, the constraint on the $\ell_2$-norm of $\hat{\mb b}$ needs to be chosen carefully with ridge regression, typically by cross-validation. In practice, the number $L$ of dimensions can be chosen as large as possible according to the available computational budget.

\subsection{Logistic regression}
We give an analogous result to Theorem~\ref{thm:main_ridge_pred_error} for classification problems under logistic loss. 
Let $\mb x_i \in [-1, 1]^{p}$ and let $\mb Y \in \{0,1\}^n$ be an associated vector of class labels. We assume the model
\begin{equation} \label{eq:main_logistic_mod}
 Y_i \sim \mathrm{Bernoulli}(p_i); \qquad \log\left( \frac{p_i}{1 - p_i} \right) = \kappa(\delta_i)\mb{x}_i ^T \mbb{\beta}^*, 
\end{equation}
with the $Y_i$ independent for $i=1,\ldots,n$. Note that we have omitted the separate intercept term for simplicity.

Here we  consider a linear classifier constructed by $\ell_2$-constrained logistic regression. One can obtain a similar result for unconstrained logistic regression based on Lemma~6.6 of \citet{buhlmann2011statistics}, but we do not pursue this further here. Define
 \begin{equation} \label{eq:logistic_def}
  \hat{\mb{b}}_\lambda = \argmin{\mb{b}} \frac{1}{n}\sum_{i=1} ^n \left[ -Y_i \mb{s}_i ^T \mb{b} + \log\{1+\exp(\mb{s}_i ^T \mb{b})\} \right] \; \text{ such that } \; \norms{\mb{b}}_2 ^2 \leq \lambda.
 \end{equation}
 Let $\mathcal{E}(\hat{\mb{b}}_\lambda)$ denote the excess risk of $\hat{\mb{b}}_\lambda$ under logistic loss, so
 \begin{equation} \label{eq:excess_risk}
  \mathcal{E}(\hat{\mb{b}}_\lambda) = \frac{1}{n} \sum_{i=1} ^n \left[ -p_i \mb{s}_i ^T \hat{\mb{b}}_\lambda + \log\{1+\exp(\mb{s}_i ^T \hat{\mb{b}}_\lambda)\} \right] - \frac{1}{n} \sum_{i=1} ^n \left[ -p_i \kappa(\delta_i)\mb{x}_i ^T \mbb{\beta}^* + \log\{1+\exp(\kappa(\delta_i)\mb{x}_i ^T \mbb{\beta}^*)\} \right].
 \end{equation}
We can now state the analogous result to Theorem~\ref{thm:main_ridge_pred_error}.
\begin{thm} \label{thm:main_logistic}
 Define $\tilde{p} \in \R$ by
 \begin{equation} \label{eq:def_p_tilde}
  \tilde{p} := \frac{1}{n}\sum_{i=1} ^n p_i (1-p_i) \leq \frac{1}{2}.
 \end{equation}
Then we have that there exists a $\lambda$ depending $\mbb\beta^*$ and $\mb S$ such that
\begin{equation*}
\E_{\mb{Y}, \mbb{\pi}, \mbb{\Psi}}\{\mathcal{E}(\hat{\mb{b}}_\lambda)\} \leq \sqrt{\frac{ 2\tilde{p}C }{(1-2^{-b})n}}\|\mbb\beta^*\|_2 + \frac{C}{2^{b+2}L (1-2^{-b})} \norms{\mbb \beta^*}_b ^2.
\end{equation*}
Here the value of  $C$ is defined as in Theorem~\ref{thm:main_OLS_pred_error} by
$ C=(2-\delta)q $ for equal row sparsity $\delta$  and  $ C = q_{\min}\log\{4\log(L)/\delta_{\min}\} $ for $\kappa=\kappa_0$ and unequal row-sparsity.
\end{thm}
The result illustrates that the usefulness of $b$-bit min-wise hashing is not limited to regression problems. In fact, most applications of are classification problems \citep{li2011theory} and our analysis of $b$-bit min-wise hashing here gives a theoretical explanation for its performance in these cases.

\section{Interaction models} \label{section:interaction}
One of the compelling aspects of regression and classification with $b$-bit min-wise hashing is the fact that a particular form of interactions between variables can be fitted. This does not require any change in the procedure other than a possible increase in $L$. To be clear, in order to capture interactions with $b$-bit min-wise hashing, just as in the main effects case, we create a reduced matrix $\mb S$ and then fit a main effects model to $\mb S$.
The dimension of the compressed data, $2^b L$, can still be substantially smaller than the $O(p^2)$ number of coefficients that would need to be estimated if the interactions were modelled in the conventional way, and so the resulting computational advantage can be very large. 

Note that in situations where the number of original predictors, $p$, may be manageable, including interactions explicitly can quickly become computationally infeasible. For example, if we start with, $10^5$ variables, the two-way interactions number more than a billion. For larger values of $p$, even methods such as Random Forest \citep{breiman01random} or Rule Ensembles \citep{friedman2005plv} would suffer similar computational problems.


We now describe a type of interaction model that can be fitted with $b$-bit min-wise hashing. Let $\mb f ^* \in \R^n$ be given by
\begin{equation} \label{eq:int_model}
 f^* _i = \sum_{k=1} ^p X_{ik} \theta^{*, (1)} _k + \sum_{k, k_1 = 1} ^p X_{ik} \ind_{\{X_{ik_1} = 0\}} \Theta^{*, (2)} _{k, k_1}, \;\; i=1, \ldots, n,
\end{equation}
where $\mbb \theta^{*, (1)} \in \R^p$ is a vector of coefficients for the main effects terms, and $\mbb \Theta^{*, (2)} \in \R^{p \times p}$ is a matrix of coefficients for interactions whose diagonal entries are zero. As elsewhere in the paper, throughout this section we will assume that $\mb X \in [-1,1]^{n\times p}$. Note that if $\mb{X}$ were a binary matrix, then \eqref{eq:int_model} parametrises (in fact over-parametrises) all linear combinations of bivariate functions of predictors; that is all possible two-way interactions are included in the model.

In general, the interaction model includes the tensor product of the set of original variables with the columns of an $n \times p$ matrix with $ik$th entry $\ind_{\{X_{ik} = 0\}}$. The value zero is thus given a special status and the model seems particularly appropriate in the sparse design setting we are considering here.

\subsection{Approximation error}
We will assume that the number of non-zero entries in each row of $\mb X$ is $q\geq 1$. However, we believe our proof techniques can be extended to the unequal sparsity and unknown row scaling scenario dealt with in Section~\ref{sec:row_scaling}. Furthermore, for technical reasons, we assume here that $p \geq 3$.

Let $\mbb\Theta^*$ collect together $\mbb\theta^{*,(1)}$ and $\mbb\Theta^{*,(2)}$ and define the following norms analogously to~\eqref{eq:norm_b}:
\begin{align}
 \| \mbb\Theta^{*} \|  := \;&  \|\mbb{\theta}^{*,(1)}\|_2 + \bigg(2(2-\delta)q\sum_{k,k_1,k_2} \abs{\Theta^{*,{(2)}}_{k k_1} \Theta^{*,{(2)}}_{k k_2}}\bigg)^{1/2},  \label{eq:intnorm1} \\
  \| \mbb\Theta^{*} \|_b  := \; & \|\mbb{\theta}^{*,(1)}\|_b  + \bigg\{2(2-\delta)q\bigg(\sum_{k,k_1,k_2} \abs{\Theta^{*,{(2)}}_{k k_1} \Theta^{*,{(2)}}_{k k_2}} + \delta (2^b-2)  \sum_{k,k_1,k_2} \frac{\|\mb X_k\|_2^2}{n} \abs{\Theta^{*,{(2)}}_{k k_1} \Theta^{*,{(2)}}_{k k_2}}\bigg)\bigg\}^{1/2}. \label{eq:intnorm2}
\end{align}

\begin{thm} \label{thm:approx_error_inter}
Suppose we have exactly $q$ non-zero entries in each row of $\mb X$. Then there exists a vector $\mb b^* \in \R^{2^b L}$ with the following properties: 
\begin{enumerate}[(i)]
 \item The approximation is unbiased, $\E_{\mbb{\pi}, \mbb{\Psi}}(\mb{S}\mb{b}^*) = \mb f^*$.
 \item The $\ell_2$-norm is bounded by
 \[
 \E_{\mbb{\pi}, \mbb{\Psi}}(\norms{\mb{b}^*}_2 ^2) \leq \frac{(2-\delta)q}{L(1-2^{-b})} \|\mbb\Theta^*\|^2.
 \]
 \item The approximation error is bounded by \begin{align*}
 &\E_{\mbb{\pi}, \mbb{\Psi}}(\norms{\mb{S}\mb{b}^* - \mb{f}^*}_2 ^2)/n \leq \frac{(2-\delta)q}{2^bL(1-2^{-b})}\|\mbb\Theta^*\|_b^2. 
 \end{align*}
\end{enumerate}
\end{thm}
The bound on the approximation error in (iii) is most suited to situations where there are a fixed number of interaction terms, so
\begin{equation} \label{eq:int-norm_fix}
\sum_{k, k_1, k_2} \abs{\Theta^{*,{(2)}} _{k k_1} \Theta^{*,{(2)}} _{k k_2}} = O(1). 
\end{equation}
Then we see that the contribution of the interaction terms to the bound on the approximation error is of order $q^2$. On the other hand, if we are considering a growing number of many small interaction terms, much tighter bounds than that given by (iii) can be obtained.
The bounds above show in particular that the form of function given by \eqref{eq:int_model} lies in the RKHS of the resemblance kernel and its RKHS norm is upper bounded by $(2-\delta)q \|\mbb\Theta^*\|^2$; further details are given in the appendix Section~\ref{sec:RKHS_interpret}.

The results for interaction models corresponding to Theorems \ref{thm:main_OLS_pred_error}, \ref{thm:main_ridge_pred_error} and \ref{thm:main_logistic} now follow.
\subsection{Prediction error}
We now present results for linear and logistic regression models where the signal involves interactions.
\subsubsection{Linear regression models}
Assume the model \eqref{eq:linmod} and define the MSPE by \eqref{eq:MSPE} but in both cases with $\mb X \mbb \beta^*$ now replaced by $\mb f^*$ \eqref{eq:int_model}. As in the previous section, we will assume that $\mb X$ has $q$ non-zero entries in each row. When OLS estimation is used, we have the following result.

\begin{thm} \label{thm:interaction_OLS_pred_error}
 Let $(\hat{\alpha}, \hat{\mb b})$ be the least squares estimator (\ref{eq:est_ols}). Then
 \begin{equation*}
  \MSPE ((\hat{\alpha}, \hat{\mb{b}})) \;\leq \; \frac{(2-\delta)q}{2^bL(1-2^{-b})}\|\mbb\Theta^*\|_b^2 + 2^b L \frac{\sigma^2}{n}.
 \end{equation*}
\end{thm}
To interpret the result, consider a situation where there are a fixed number of interaction and main effects of fixed size, so in particular \eqref{eq:int-norm_fix} holds. Then treating $b$ as fixed, the optimal $L$, $L^*=O(\sqrt{q^2 n/\sigma})$. If $n, q $ and $p$ increase by collecting new data and adding uninformative variables, then in order for the MSPE to vanish asymptotically, we require $q^2 / n \to 0$. Compare this to the corresponding requirement of OLS applied to $\mb X$, that $p^2 / n \to 0$. Particularly in situations of increasing variable sparseness, as discussed in Section~\ref{sec:main_OLS_MSPE}, this can amount to a large statistical advantage.

The computational gains can be equally great.
If, for example, $n \approx q^2$, then $L^*=O(q^2)$. If ridge regression were applied to $\mb X$ augmented by $O(p^2)$ interaction terms, the number of operations required would be $O(p^2 q^4)$; OLS using $\mb S$ has complexity $O(q^6)$. If instead $n \approx p^2$, then regression with explicitly coded interaction terms would have complexity $O(p^6)$, whilst with the compressed data this would be reduced to $O(p^4 q^2)$. 

As in the main effects case, the ridge regression result is similar.
\begin{thm} \label{thm:interaction_ridge_pred_error}
Let the ridge regression estimator be given by \eqref{eq:main_ridge_def}.
There exists $\lambda$ depending on $\mb f^*$ and $\mb S$ such that
we have
 \begin{align*}
  \MSPE ((\hat{\alpha}, \hat{\mb{b}})) \;\leq \; & \sigma\sqrt{\frac{(2-\delta)q}{n(1-2^{-b})}} \|\mbb\Theta^*\| + \frac{(2-\delta)q}{2^bL(1-2^{-b})}\|\mbb\Theta^*\|_b^2  + \frac{\sigma^2}{n}.
 \end{align*}
\end{thm}
Similarly to Theorem~\ref{thm:main_ridge_pred_error} the result here suggests choosing a large $L$ is always better from a statistical point of view. However,  for computational reasons, it may not be possible to take $L$ much larger than $L^*$.

\subsubsection{Logistic regression}
Here we assume the model~\eqref{eq:main_logistic_mod} and define the excess risk by \eqref{eq:excess_risk}, but in both cases with $\mb X \mbb \beta^*$ replaced by $\mb f^*$.
\begin{thm} \label{thm:interaction_logistic}
Define $\tilde{p} \in \R$ as in \eqref{eq:def_p_tilde} and the $\ell_2$-penalised logistic regression estimator as in \eqref{eq:logistic_def}.
Then we have that there exists $\lambda$ such that
\begin{align*}
\E_{\mb{Y}, \mbb{\pi}, \mbb{\Psi}}\{\mathcal{E}(\hat{\mb{b}}_\lambda)\} &\leq  
\sigma\sqrt{\frac{\tilde{p}(2-\delta)q}{n(1-2^{-b})}} \|\mbb\Theta^*\|+\frac{(2-\delta)q}{2^{b+2}L(1-2^{-b})}\|\mbb\Theta^*\|_b^2.
\end{align*}
\end{thm}

One could continue to look at higher-order interaction models by adding three-way interactions in \eqref{eq:int_model} and adapting \eqref{eq:intnorm1} and \eqref{eq:intnorm2} in suitable ways. However, being able to show that two-way interaction models can be fitted with $b$-bit min-wise hashing may well be sufficient for most applications.

\section{Extensions}
\label{section:extensions}

\subsection{Variable importance} \label{sec:var_imp}
Typically prediction, rather than model selection, is the primary goal in large-scale applications with sparse data, one reason for this being that we cannot  expect a very small subset of variables to approximate the signal well when the design matrix is sparse. Nevertheless, it is often illuminating to study the influence of specific variables or look for the variables that have the largest influence on predictions. Indeed, such study is often undertaken following applications of Random Forest \citep{breiman01random}, where several variable importance measures allow practitioners to better interpret the fits produced.

We now describe how importance measures can be obtained for $b$-bit min-wise hashing as described in Section~\ref{sec:shuffled}. Let $\hat{f}:\R^p \to \R$ be the regression function created following regression on $b$-bit min-wise hashed data, and let $\hat{f} _i := \hat{f}(\mb x_i)$. Furthermore, for $k=1,\ldots, p$, let $\hat{f}^{(-k)}:= \hat{f}(\mb x_i ^{(-k)})$, where $\mb x_i ^{(-k)}$ is equal to $\mb x_i$ but with $k$th component set to zero.

The vector $\hat{\mb f} - \hat{\mb f}^{(-k)}$ is the difference in predictions obtained when fitting to $\mb X$, and those obtained when fitting to $\mb X$ with the $k$th column set to zero.
When the underlying model in $\mb X$ contains only main effects \eqref{eq:linmod} and no structural error is present, we might expect that
\[
 \hat{\mb f} - \hat{\mb f}^{(-k)} \approx \mbb \beta_k ^* \mb X_k.
\]
To obtain a measure of variable importance, one could look at the $\ell_2$-norm of $\hat{\mb f} - \hat{\mb f}^{(-k)}$, for example \citep{breiman01random}.


The difference in predictions can be computed relatively easily by considering the $n \times 2^b L$ matrix $\tilde{\mb S}$ with entries given by  $\tilde{S}_{ilc} = \tilde{S}_{i(c+(l-1)2^b}= X_{i \tilde{H}_{il}} \ind_{\{\Psi_{\tilde{H}_{il} l}=c\}}$, where
\[
 \tilde{H}_{il} := \argmin{k \in \mb z_i \setminus H_{il}} \pi_l (k).
\]
Thus $\tilde{H}_{il}$ is the variable index in $\mb z_i$ whose value under permutation $\pi_l$ is second smallest among $\{\pi_l(k):k \in \mb z _i\}$. If $\mb z_i \setminus H_{il} = \emptyset$, we simply set $\tilde{S}_{il}=0$.
Then
\begin{equation} \label{eq:varinfluence}
 \hat{f}_i - \hat{f}^{(-k)} _i = \sum_{l=1} ^L \ind_{\{H_{il}=k\}} \sum_{c=1}^{2^b} (S_{ilc} -\tilde{S}_{ilc})  \hat{b} _{lc}.
\end{equation}
Note that we only need to store the $n\times L$ matrix $\mb H$ and $n \times 2^bL$ matrices $\mb S$ and $\tilde{\mb S}$ to compute the variable importance for all variables; moreover the latter matrices only have at most $nL$ non-zero entries each.

Interaction effects are not directly visible, but do manifest themselves in the form of a higher variability among $\{\hat{f}_i - \hat{f}^{(-k)} _i : \mb x_i \approx \mb x\}$, for any given value of $\mb x$, if variable $k$ is involved in an interaction term. In principle, one could attempt to detect this increased variability, but further investigation of this is beyond the scope of the current work. 


\subsection{Other fitting procedures}
Here we have only considered OLS, ridge regression and $\ell_2$-penalised logistic regression as prediction methods after reducing the design matrix.
However, it is also conceivable that other fitting procedures could be suitable. In particular, it would be interesting to look at matching pursuit, boosting and the Lasso, for which results in \citep{tropp2004gga,buhlmann2006bhd,van2008high} could be leveraged. Matching pursuit would have the computational advantage that the entire $\mb S$ matrix would not need to be held in memory. Instead, one could create the columns  during the fitting process. Such an approach may be useful for problems where the dimension of the hashing-matrix, $2^b L$, needs to be very large to achieve a desired predictive accuracy.

\section{Discussion} \label{sec:discussion}
In this paper we have derived approximation error bounds for $b$-bit min-wise hashing. We were able to show that not only does $b$-bit min-wise hashing take advantage of sparsity in the design matrix computationally, it is also able to exploit this for improved statistical performance. In particular, the MSPE of regression following dimension reduction by $b$-bit min-wise hashing is of the form $\sqrt{q/n} \|\mbb \beta^*\|_2$ if the data follow a linear model with coefficient vector $\mbb \beta^*$ and $q$ is the average number of non-zero variables for an
observation.
The linear model can then be well-approximated by the low-dimensional $b$-bit min-wise hashed data if the norm of $\|\mbb \beta^*\|_2$ is low, as occurs, for example if the signal is approximately replicated in distinct blocks of variables.

In addition, we have shown that more complicated models such as interaction models can be fitted by a regression on the hashed data matrix that contains only main effects. Though a larger dimension  $L$ of the hashed data may be required than when approximating a main effects model, no further changes are needed to the procedure.

These bounds also reveal some of the predictive properties of the resemblance kernel, and provide an insight into the sorts of regression functions that have small norm in its associated RKHS. More generally, we believe that random feature expansions may well be useful as a theoretical tool to understand properties of otherwise intractable kernels. We expect to see more extensions and applications $b$-bit min-wise hashing and other random feature expansions, both as computational and theoretical tools, in the future.

\appendix


\section{Approximation error results} \label{app:approx}
In this section we prove results on the approximation error presented in the main text (Theorems~\ref{thm:approx_error_main_binary}, \ref{thm:scaling_adapt} and \ref{thm:approx_error_inter}) as well as an additional result on the approximation error of linear signals when row sparsity is not necessarily equal (Theorem~\ref{thm:unequal_sparsity}).
\subsection{Preliminary results}
We will let $q_i$ be the number of non-zeroes in the $i$th row of $\mb X$ and define $\delta_i=q_i/p$. We will assume that $q_i \geq 1$ for all $i$.
For the proofs of results on approximation error in settings with just main effects, we will make use of the following lemma. This lemma formalises the ideas of the discussion at the end of Section~\ref{sec:row_scaling}, that the elements of $\mb M$ behave rather like geometric random variables. 

\begin{lem} \label{lem:geometric_equiv}
There exist random functions $\{g_{l}(k)\}_{l=1, \ldots, L, \, k=1, \ldots, p}$ defined on the same probability space as the permutations $\mbb \pi$ with the following properties:
\begin{enumerate}[(i)]
\item The random variables $\{g_1(k)\}_{k=1, \ldots,p}, \ldots, \{g_{L}(k)\}_{k=1, \ldots,p}$ are i.i.d. and are independent of $\mbb\Psi$. 
\item The rank of $g_{l}(k)$ among $g_l(1), \ldots, g_l(p)$ taken in increasing order is $\pi_l(k)$.
\item Marginally $g_l(k) \sim \text{Geo}(p^{-1})$.
\item $G_{il}:=\min_{k \in \mb z_i}g_l(k) = g_l(H_{il}) \sim \text{Geo}(\delta_i)$.
\item $\mb G$ and $\mb H$ are independent.
\end{enumerate}
\end{lem}
\begin{proof}
First consider generating permutations $\mbb\pi$ in the following way. Let $m \in \mathbb{N}$ and let $\sigma_1^{(m)},\ldots, \sigma_L^{(m)}$ be $L$ i.i.d.\ random permutations of $\{1, \ldots, mp\}$. For $k=1,\ldots,p$, let
\[
g_l^{(m)}(k) = \min_{a=0, \ldots,m-1} \sigma_l^{(m)}(k + ap).
\]
Note that the $g_l^{(m)}(k)$ are all distinct and any ordering of them is equally likely so they define a random permutation of $\{1,\ldots,p\}$. Furthermore, for $j=1, \ldots, mp-m+1$,
\begin{align*}
\pr(g_l^{(m)}(k) = j) =  \binom{mp-j}{m-1} \bigg/ \binom{mp}{m} = \frac{1}{p}\bigg(1-\frac{1-m^{-1}}{p-m^{-1}}\bigg)\cdots\bigg(1-\frac{1-m^{-1}}{p-(j-1)m^{-1}}\bigg).
\end{align*}
Thus
\[
\pr(g_l^{(m)}(k) = j) \to \frac{1}{p}\bigg(1-\frac{1}{p}\bigg)^{j-1}
\]
as $m \to \infty$ for $j=1, 2, \ldots$. Similarly $G_{il}^{(m)}:=\min_{k \in \mb z_i}g_l^{(m)}(k)$ has $\pr(G^{(m)}_{il}=j)\to \delta_i(1-\delta_i)^{j-1}$ as $m \to \infty$. Note that $\mb G^{(m)}$ and $\mb H$ are independent. Thus
\[
\{g^{(m)}_l(k)\}_{l=1, \ldots,L, k=1,\ldots,p} \indist \{g_l(k)\}_{l=1, \ldots,L, k=1,\ldots,p} 
\]
as $m \to \infty$ with the random variables $g_l(k)$ having the properties given in the statement of the lemma.
\end{proof}

In the proofs which follow, we will consider the permutations as having been generated as described by Lemma~\ref{lem:geometric_equiv}. We will let $\pi = \pi_1$, $M_i=M_{i1}$, $g=g_1$, $G_1=G_{i1}$, $H_i=H_{i1}$ and $\mbb\psi=\mbb\Psi_1$. Let $C=2^b$, $\nu =2^{-b}$.

The next lemma introduces the general form of $\mb b^*$ that we will use for the main effects results. It also establishes results on the mean and variance of the approximation and gives a bound on $\E(\|\mb b^*\|_2^2)$; these will form the basis of the theorems to follow.

\begin{lem}\label{lem:main_approx}
For a given sequence of weights $\{w_j\}_{j=1}^\infty$, let $\tilde{\mb b}^* \in \R^{LC}$ be given by
\[
\tilde{b}^*_{lc}=\frac{1}{L}\sum_{k=1}^p \beta^*_k \frac{\ind_{\{\Psi_{lk}=c\}}-\nu}{1-\nu} w_{g_l(k)}
\]
and let $\mb b^* = \E(\tilde{\mb b}^*| \mbb\pi)$. We have the following.
\begin{enumerate}[(i)]
\item \[
\E_{\mbb\pi, \mbb\Psi} (\mb s_i^T \mb b^*)=\frac{1}{p} \mb x_i^T \mbb\beta^* \sum_{\ell=1}^\infty(1-\delta_i)^{\ell-1}w_{\ell} .
\]
\item \begin{align}
\E_{\mbb\pi, \mbb\Psi}(\|\mb b^*\|_2^2) \leq \frac{1}{pL(1-\nu)} \|\mbb\beta^*\|_2^2 \sum_{\ell=1}^\infty w_\ell^2. \label{eq:norm_gen}
\end{align}
\item \begin{align}
\Var_{\mbb\pi, \mbb\Psi}(\mb s_i^T \mb b^*) \leq \frac{1}{pL(1-\nu)}\bigg(\nu\|\mbb\beta^*\|_2^2 + (1-2\nu)\sum_{k=1}^pX_{ik}^2{\beta^*_k}^2 \bigg)\sum_{\ell=1}^\infty w_\ell^2. \label{eq:var_gen}
\end{align}
\end{enumerate}
\end{lem}
\begin{proof}
First note that
\begin{align}
\E\bigg( \frac{\ind_{\{\psi_{k}=\psi_{j}\}} -\nu}{1-\nu}\bigg|
 \psi_{j}\bigg) &= \begin{cases} 1 & \mbox{if } k=j \\
     0 & \mbox{otherwise} \end{cases} \label{eq:ind1}\\
     \E\bigg( \frac{\ind_{\{\psi_{k}=\psi_{j}\}} -\nu}{1-\nu} \frac{\ind_{\{\psi_{\ell}=\psi_{j}\}} -\nu}{1-\nu}\bigg|
 \psi_{j}\bigg) &= \begin{cases} 1 & \mbox{if } k=\ell=j \\
     0 & \mbox{if } k \neq \ell \\
     \frac{\nu}{1-\nu} & \mbox{otherwise.}\end{cases} \label{eq:ind2}
\end{align}
For (i), we have
\begin{align*}
\E_{\mbb\pi, \mbb\Psi} (\mb s_i^T \mb b^*) &= \E_{g, \mbb\psi} \bigg(\sum_{c=1}^C\sum_{j = 1} ^p X_{ij}\ind_{\{H_i = j\}} \ind_{\{\psi_j=c\}} \sum_{k=1}^p \beta^*_k \frac{\ind_{\{\psi_k=c\}}-\nu}{1-\nu} w_{g(k)}\bigg) \\
&=  \E_{g} \bigg(\sum_{k = 1} ^p X_{ik}\ind_{\{H_i = k\}} \beta^*_k w_{g(k)}\bigg) \\
&=  \frac{1}{q_i} \sum_{k = 1} ^p X_{ik} \beta^*_k \E(w_{G_i}),
\end{align*}
where to arrive at the second line we used \eqref{eq:ind1}.

Turning to (ii), note that each component of $\mb b^*$ has mean zero and so
\[
\E({b_{lc}^*}^2)=\Var(b_{lc}^*)=\Var\{\E(\tilde{b}_{lc}^*| \mbb\pi)\} \leq \Var(\tilde{b}_{lc}^*).
\]
Now we have
\begin{align*}
\E_{g_1,\ldots,g_L, \mbb\Psi}\|\tilde{\mb b}^*\|_2^2 = \frac{1}{L}\sum_{c=1}^C \sum_{k, \ell}\beta^*_k\beta^*_\ell \E\bigg(\frac{\ind_{\{\psi_k=c\}}-\nu}{1-\nu}\frac{\ind_{\{\psi_\ell=c\}}-\nu}{1-\nu}\bigg)\E(w_{g(k)}w_{g(\ell)})
\end{align*}
Using \eqref{eq:ind2}, we get
\begin{align*}
\E_{g_1,\ldots,g_L, \mbb\Psi}\|\tilde{\mb b}^*\|_2^2 &=\frac{1}{L(1-\nu)}\sum_{k}{\beta^*_k}^2 \E(w_{g(k)}^2) \leq \frac{1}{pL(1-\nu)} \|\mbb\beta^*\|_2^2 \sum_{\ell=1}^\infty w_\ell^2.
\end{align*}

For (iii) we argue as follows.
\begin{align*}
\Var(\mb s_i^T \mb b^*)  &\leq \Var(\mb s_i^T \tilde{\mb b}^*)\\
&\leq \frac{1}{L}\E_{g,\mbb \psi} \bigg( X_{i H_i} ^2 \sum_{k, \ell} \beta^* _k \beta^* _\ell  \frac{\ind_{\{\psi_{k}=\psi_{H_i}\}} -\nu}{1-\nu} \frac{\ind_{\{\psi_{\ell}=\psi_{H_i}\}} -\nu}{1-\nu} w_{g(k)} w_{g(\ell)} \bigg)
\end{align*}
Using \eqref{eq:ind2} and the fact that $\mb X \in [-1,1]^{n \times p}$, we have
\begin{align}
 \Var(\mb s_i^T \mb b^*) &\leq \frac{1}{L}\E\bigg\{X_{iH_i}^2\bigg(\frac{\nu}{1-\nu}\sum_{k=1}^p(\beta^*_k)^2w_{g(k)}^2 + \frac{1-2\nu}{1-\nu}(\beta_{H_i}^*)^2w_{G_i}^2\bigg)\bigg\} \label{eq:equal_q_var} \\
 &\leq \frac{1}{L(1-\nu)}\bigg\{\nu\sum_{k=1}^p{\beta^*_k}^2\E(w^2_{g(k)}) + \frac{1-2\nu}{q_i}\E(w_{G_i}^2)\sum_{k=1}^pX_{ik}^2{\beta^*_k}^2 \bigg\}. \notag
\end{align}
The result then follows as
\[
 \sum_{\ell=1}^\infty w_{\ell}^2 \geq \E(w^2_{g(k)}) = \frac{1}{p}\sum_{\ell=1}^\infty w_{\ell}^2 \bigg(1-\frac{1}{p}\bigg)^{\ell-1} \geq \frac{\delta_i}{q_i}\sum_{\ell=1}^\infty w_{\ell}^2(1-\delta_i)^{\ell-1} = \frac{\E(w_{G_i}^2)}{q_i}. \qedhere
\]
\end{proof}

\subsection{Proof of Theorem~\ref{thm:approx_error_main_binary}}
We use a $\mb b^*$  and $\tilde{\mb b}^*$ as in Lemma~\ref{lem:main_approx} but here we choose the weights $w_\ell$ so as to minimise $\sum_{\ell=1}^\infty w_{\ell}^2$ (a term which features in our upper bounds on the variance and $\E(\|\mb b^*\|_2^2)$) subject to the unbiasedness constraint (i). The unbiasedness constraint amounts to
\[
\sum_{\ell=1}^\infty (1-\delta)^{\ell-1}w_{\ell}=p.
\]
Performing the minimisation with this constraint yields
\[
w_\ell = p \frac{(1-\delta)^{\ell-1}}{\sum_{\ell=1}^\infty (1-\delta)^{2\ell-2}}.
\]
With this choice we have
\begin{align*}
\sum_{\ell=1}^\infty w_\ell^2 = p^2 \bigg(\sum_{\ell=1}^\infty (1-\delta)^{2\ell-2}\bigg)^{-1} = p^2\{1-(1-\delta)^2\} = (2-\delta)qp.
\end{align*}
Substituting into \eqref{eq:norm_gen} and \eqref{eq:var_gen} then yields the result. \qed

\subsection{Proof of Theorem~\ref{thm:scaling_adapt}}
We use a $\mb b^*$  and $\tilde{\mb b}^*$ as in Lemma~\ref{lem:main_approx} but here we take
\[
 w_{\ell+1} = p(-1)^{\ell}\frac{\kappa^{(\ell)}(1)}{\ell!}\{\ind_{\{\ell\leq \floor{m}\}}+ (m-\floor{m})\ind_{\{\ell=\ceil{m}\}}\}
\]
where $m >0$ is a parameter to be chosen. Thus the weights correspond to coefficients from a truncated Taylor series expansion of $\kappa$ about 1.  We have
\begin{equation*}
\E_{\mbb\pi, \mbb\Psi}[\{ (\delta_{\min}/\delta_i)^a\mb x_i^T\mbb\beta^* - \mb s_i^T \mb b^*\}^2] =\{ (\delta_{\min}/\delta_i)^a\mb x_i^T\mbb\beta^* - \E_{\mbb\pi, \mbb\Psi}(\mb s_i^T \mb b^*)\}^2 +  \Var_{\mbb\pi,\mbb\Psi}(\mb s_i^T \mb b^*).
\end{equation*}
We first bound the variance term by bounding the squared sum of the sequence of weights. 
To this end, we note that by Lemma~\ref{lem:taylor_coef_bd}
\begin{align*}
 \frac{\delta_{\min}^{-2a}}{p^2}\sum_{\ell=1}^\infty w_\ell^2 \leq 1 + a^2 + a^2 e^{2a} \bigg(\sum_{\ell=2}^{\floor{m}}\frac{1}{\ell^{2(1-a)}}   + \frac{m-\floor{m}}{\ceil{m}^{2(1-a)}}\bigg).
\end{align*}
Now
\begin{align*}
\sum_{\ell=2}^{\floor{m}} \frac{1}{\ell^{2(1-a)}}   + \frac{m-\floor{m}}{\ceil{m}^{2(1-a)}} &\leq \int_{1}^m \frac{1}{\ell^{2(a-1)}}d\ell \\
&= \begin{cases}
\frac{m^{2a-1}-1}{2a-1} \qquad &\text{if } a \neq 1/2 \\
\log(m) \qquad &\text{if } a = 1/2.
\end{cases}
\end{align*}
Let
\[
\tau_a(m) = \begin{cases} e\log(me^{5/e})/4 \qquad& \text{if } a=1/2, \\
a^2e^{2a}m^{2a-1}/(2a-1)  \qquad & \text{if }1/2< a\leq 1.
\end{cases}
\]
Then
\begin{equation} \label{eq:w_sq_a_bd}
\sum_{\ell=1}^\infty w_{\ell}^2 \leq p^2 \delta_{\min}^{2a}\tau_a(m).
\end{equation}
The variance is then at most
\begin{align*}
\delta_{\min}^{2a}\tau_a(m) \frac{p}{L(1-\nu)}\bigg(\nu\|\mbb\beta^*\|_2^2 +(1-2\nu)\sum_{k=1}^p X_{ik}^2{\beta^*_k}^2 \bigg).
\end{align*}
Turning now to the bias term, note first that by (i) of Lemma~\ref{lem:main_approx}, this is equal to
\begin{align} \label{eq:bias}
(\mb x_i^T\mbb\beta^*)^2\bigg\{(\delta_{\min}/\delta_i)^a-\frac{1}{p}\sum_{\ell=1}^\infty(1-\delta_i)^{\ell-1}w_\ell\bigg\}^2.
\end{align}
We see this is bounded above by
\begin{align*}
 \delta_{\min}^{2a}(\mb x_i^T\mbb\beta^*)^2 \bigg\{ae^a \bigg(\sum_{\ell=\ceil{m}}^\infty (1-\delta_i)^\ell \frac{1}{\ell^{1-a}}\bigg)\bigg\}^2.
\end{align*}
Now
\[
 \sum_{\ell=\ceil{m}}^\infty (1-\delta_i)^\ell \frac{1}{\ell^{1-a}} \leq \frac{e^{-\delta_i m}}{m^{1-a}\delta_i}.
\]
By the Cauchy--Schwarz inequality (assuming $X_{ij} \in [-1,1]$)
\begin{align*}
\frac{(\mb x_i^T\mbb\beta^*)^2}{\delta_i}  = \frac{1}{\delta_i}\bigg(\sum_{k \in \mb z_i} X_{ik}\beta^*_k\bigg)^2 \leq p \sum_{k=1}^pX_{ik}^2{\beta^*_k}^2 \leq p\|\mbb\beta^*\|_2^2.
\end{align*}
Thus the squared bias is at most
\[
\frac{p}{1-\nu} \frac{a^2e^{2a}}{m^{1-2a}}\max_{i=1,\ldots,n}\bigg(\frac{e^{-2\delta_i m}}{m\delta_i}\bigg)\bigg(\nu\|\mbb\beta^*\|_2^2 +(1-2\nu)\sum_{k=1}^p X_{ik}^2{\beta^*_k}^2 \bigg).
\]
Therefore the MSE (now averaging over the observations) is bounded by the minimum over $m>0$ of
\begin{gather*}
  \frac{p}{L (1-2^{-b})}\delta_{\min}^{2a}\bigg\{ \tau_{a}(m) + \frac{a^2e^{2a}}{m^{1-2a}} \max_{i=1,\ldots,n}\bigg(\frac{e^{-2\delta_i m}}{m\delta_i}\bigg)\bigg\} \|\mbb\beta^*\|_b^2.
\end{gather*}
For $a=1/2$, we set $m=\log(L)/\{2\delta_{\min}\}$. This yields 
 \begin{align*}
 \min_{m>0} \bigg\{\tau_{1/2}(m) + \frac{Le}{4}\max_{i=1,\ldots,n}\bigg(\frac{e^{-2\delta_i m}}{m\delta_i}\bigg)\bigg\} &\leq 
 \frac{e}{4}\bigg\{\log\bigg(\frac{\log(L)e^{5/e}}{2\delta_{\min}}\bigg) + \frac{2}{\log(L)}\bigg\} \\
 &\leq \log\{4\log(L)/\delta_{\min}\}
 \end{align*}
 provided $L \geq 10$ and $\delta_{\min} \leq 1/2$.
 Finally the bound for $a>1/2$ comes from setting
 \[
 m= \frac{1}{2}\log\{2(2a-1)L\}/\delta_{\min}
 \]
which gives
\begin{align*}
\min_{m>0}\bigg\{\tau_a(m) + \frac{La^2e^{2a}}{m^{1-2a}}\max_{i=1,\ldots,n}\bigg(\frac{e^{-2\delta_i m}}{m\delta_i}\bigg) \bigg\} &\leq \frac{ \delta_{\min}^{1-2a}a^2e^{2a}}{2^{2a-1}(2a-1)}  [\log\{2(2a-1)L\}]^{2a-2}\log\{2(2a-1)eL\} \\
& \leq \frac{4\delta_{\min}^{1-2a}}{1-2a} [\log\{2(2a-1)L\}]^{2a-1}
\end{align*}
for $L \geq 2/(1-2a)$.
Using the bounds on $\tau_a$ with these choices of $m$ and \eqref{eq:w_sq_a_bd}, we obtain the bounds on $\E(\|\mb b^*\|_2^2)$ by substituting into \eqref{eq:norm_gen}.
\qed

\subsection{Unequal row sparsity and constant row-scaling} \label{sec:unequal}
Here we prove results indicated after the presentation of Theorem~\ref{thm:approx_error_main_binary} in Section~\ref{sec:un-scaled}.
 When the scaling function is simply the constant 1, the spread of the $\delta_i$ becomes more critical in determining how well the signal can be approximated. Define
\begin{align*}
\bar{\delta}&=\frac{1}{n}\sum_{i=1}^n \delta_i, \\
\mathcal{V}(\mbb\delta)&= \frac{1}{\norms{\mb X\mbb\beta^*}_2^2}\sum_{i=1}^n (\mb x_i^T \mbb\beta^*)^2(\delta_i - \bar{\delta})^2.
\end{align*}
\begin{thm} \label{thm:unequal_sparsity}
Suppose
\begin{equation} \label{eq:L_bd}
2^bL (1-2^{-b}) \leq \dfrac{p(2\bar{\delta})^3\|\mbb\beta^*\|_b^2}{\norms{\mb X\mbb\beta^*}_2^2\mathcal{V}(\mbb\delta)/n}.
\end{equation}
Then there exists $\mb b^* \in \R^L$ such that the approximation error satisfies 
\begin{equation} \label{eq:unequal_q_approx}
\frac{1}{n}\E_{\mbb\pi, \mbb\Psi}\{\norms{\mb X\mbb\beta^* - \mb S \mb b^*}_2^2\} \leq \frac{6p\bar{\delta}}{2^b L(1-2^{-b})}\|\mbb\beta^*\|_b^2,
\end{equation}
and
\begin{equation} \label{eq:unequal_q_b}
\E_{\mbb\pi, \mbb\Psi}(\|\mb b ^*\|_2^2) \leq \frac{2\bar{q}}{L(1-2^{-b})}\|\mbb\beta^*\|_2^2.
\end{equation}
\end{thm}
Provided $2^b L$ is not too large, we recover essentially the same approximation error bound as
Theorem~\ref{thm:approx_error_main_binary} up to a constant factor, but with the row sparsity replaced by the average row sparsity $\bar{\delta}$. In the simple situation where the entries of $\mb X$ are realisations of i.i.d.\ Bernoulli random variables with probability $\delta$, we would have $\bar{\delta} \approx \delta$,
$\norms{\mb X\mbb\beta^*}_2^2/n \approx \delta \norms{\mbb\beta^*}_2^2$ and
$\mathcal{V}(\mbb\delta)\approx \delta/p$. Substituting these values into the requirement on $2^b L$ shows that the condition reduces to $2^b L \leq 8p^2 \delta\{1+(2^b-2)\delta\}$.
Note that typically one would choose $2^bL$ of the order $\bar{\delta}p$. More generally, provided $\mathcal{V}(\mbb\delta)$ and
$\|\mb X\mbb\beta^*\|_2^2 / \|\mbb\beta^*\|_2^2$ are small, we can expect that the bound of Theorem~\ref{thm:approx_error_main_binary} will hold true, up to a constant factor. 

\subsubsection*{Proof of Theorem~\ref{thm:unequal_sparsity}}
We use a $\mb b^*$  and $\tilde{\mb b}^*$ as in Lemma~\ref{lem:main_approx} taking
\[
w_\ell = p(1-\bar{\delta})^{\ell-1}\ind_{\{\ell\leq m\}}\frac{\bar{\delta}(2-\bar{\delta})}{1-(1-\bar{\delta})^{2m}}.
\]
where $m\in \mathbb{N}$ is a parameter to be chosen.  This gives
\begin{align*}
 \frac{1}{p^2}\sum_{\ell=1}^\infty w_{\ell}^2 = \frac{\bar{\delta}(2-\bar{\delta})}{1-(1-\bar{\delta})^{2m}},
\end{align*}
which gives us a bound on the variance term.

Lemma~\ref{lem:main_approx} (i) gives the expression for the bias term.
To bound this, first note that
\begin{align*}
\frac{1}{p}\sum_{\ell=1}^m(1-\bar{\delta})^{\ell-1}w_{\ell}=1.
\end{align*}
Next
\begin{align*}
\bigg[\sum_{\ell=1}^m(1-\bar{\delta})^{\ell-1}\{(1-\bar{\delta})^{\ell-1}-(1-\delta_i)^{\ell-1}\}\bigg]^2 &= (\delta_i-\bar{\delta})^2\bigg[\sum_{\ell=1}^m(1-\bar{\delta})^{\ell-1}\sum_{k=0}^{\ell-2}(1-\bar{\delta})^k(1-\delta_i)^{\ell-2-k}\bigg]^2 \\
&\leq (\delta_i-\bar{\delta})^2\bigg(\sum_{\ell=1}^m(1-\bar{\delta})^{\ell-1}(\ell-1)\bigg)^2 \\
&=\min\bigg\{\frac{m(m-1)}{2}, \, \frac{1}{\bar{\delta}^2}\bigg\}^2(\delta_i-\bar{\delta})^2.
\end{align*}
Also note that as
\[
(1-\bar{\delta})^{2m} \leq 1-2m\bar{\delta} + m(2m-1)\bar{\delta}^2
\]
we have
\begin{align*}
\frac{\bar{\delta}(2-\bar{\delta})}{1-(1-\bar{\delta})^{2m}} &\leq \frac{2}{2m-m(2m-1)\bar{\delta}}\ind_{\{m\leq 1/(2\bar{\delta})\}}  + \frac{2}{1/\bar{\delta} - (1/\bar{\delta}-1)/2}\ind_{\{m > 1/(2\bar{\delta})\}}\\
&\leq \max\bigg(\frac{2}{m+1/2},\, \frac{4\bar{\delta}}{1+\bar{\delta}}\bigg)
\end{align*}
and for $m\leq 1/(2\bar{\delta}) + 1/2$,
\[
\frac{m(m-1)}{2} \max\bigg(\frac{2}{m+1/2},\, \frac{4\bar{\delta}}{1+\bar{\delta}}\bigg) \leq (m-1/2)\ind_{\{1 < m \leq 1/(2\bar{\delta}) + 1/2\}}.
\]
Thus the overall approximation error is bounded above by the minimum over $m=1, 2, \ldots, \floor{1/(2\bar{\delta}) + 1/2}$ of
\[
\ind_{\{m>1\}}(m-1/2)^2\frac{1}{n}\norms{\mb X\mbb\beta^*}_2^2\mathcal{V}(\mbb\delta) + \max\bigg(\frac{2}{m+1/2},\, 4\bar{\delta}\bigg)\frac{p}{2^b L(1-2^{-b})}\|\mbb\beta^*\|_b^2,
\]
which in turn is bounded by the minimum over $m \in [0,1/(2\bar{\delta})]$ of
\begin{equation} \label{eq:m_bd}
m^2\frac{1}{n}\norms{\mb X\mbb\beta^*}_2^2\mathcal{V}(\mbb\delta) + \frac{2}{m}\frac{p}{2^b L(1-2^{-b})}\|\mbb\beta^*\|_b^2.
\end{equation}
 Optimising over $m>0$ in the above then gives
 \[
 m=\min\bigg\{\bigg(\frac{p\|\mbb\beta^*\|_b^2}{2^b L (1-2^{-b})\|\mb X\mbb\beta^* \|_2^2\mathcal{V}(\mbb\delta)/n}\bigg)^{1/3}, \, \frac{1}{2\bar{\delta}}\bigg\}.
 \]
The condition on $L$ \eqref{eq:L_bd} ensures that the minimum is achieved at $1/(2\bar{\delta})$. Substituting this value of $m$ into \eqref{eq:m_bd} then gives \eqref{eq:unequal_q_approx}. For \eqref{eq:unequal_q_b} we note that
\[
\sum_{\ell=1}^\infty w_\ell^2 \leq 2p^2 \bar{\delta},
\]
and use Lemma~\ref{lem:main_approx} (ii).
\qed

\subsection{Proof of Theorem~\ref{thm:approx_error_inter}} \label{proof:thm:approx_error_inter}
We let $\mb b^* = \mb b^{*,(1)} + \mb b^{*,(2)}$ where $\mb b^{*,(1)}$ is chosen in line with Theorem~\ref{thm:approx_error_main_binary}. Explicitly, let $\mb b^{*,(1)} = \E(\tilde{\mb b}^*| \mbb \pi)$ where
\[
\tilde{b}^*_{lc}=\frac{p}{L}\sum_{k=1}^p \theta^{*,(1)}_k \frac{\ind_{\{\Psi_{lk}=c\}}-\nu}{1-\nu}
\frac{(1-\delta)^{g_l(k)-1}}{\sum_{\ell=1}^\infty (1-\delta)^{2\ell-2}}.
\]
We construct $\mb b^{*,(2)}$ to approximate the interactions as follows. Let
\begin{gather*}
 b_{lc}^{*,(2)} = \frac{pq}{L} \sum_{k=1} ^p \frac{\ind_{\{\Psi_{lk}=c\}} -\nu}{1-\nu} \sum_{k_1 = 1} ^p \Theta^{*,{(2)}} _{k k_1} \ind_{\{\pi_l (k_1) < \pi_l (k)\}} w_{\pi_l(k)},
\end{gather*}
where $\mb w \in \R^p$ is a vector of weights to be chosen such that
\begin{equation} \label{eq:unbiased_inter}
\E_{\mbb{\pi}, \mbb{\Psi}}(\mb{s}_i^T\mb{b}^{*,(2)}) = \sum_{k, k_1} X_{ik}\ind_{\{X_{ik_1} = 0\}}\Theta^{*,{(2)}} _{k k_1}.
\end{equation}
We compute
\begin{align*}
\E_{\mbb{\pi}, \mbb{\Psi}}(\mb{s}_i^T\mb{b}^{*,(2)}) &= \frac{pq}{L} \sum_{l=1} ^L \sum_{c=1}^C \E_{\pi_l, \mbb{\Psi}_l} \left( {S}_{ilc} \sum_{k=1} ^p \frac{\ind_{\{\Psi_{kl}=c\}} -\nu}{1-\nu} \sum_{k_1 = 1} ^p \Theta^{*,{(2)}}_{k k_1} \ind_{\{\pi_l (k_1) < \pi_l (k)\}} w_{\pi_l (k)} \right) \\
  & = pq \E_{\pi, \mbb{\psi}} \left( \sum_{c=1}^C \sum_{j = 1}^p X_{ij} \ind_{\{H_i = j\}} \ind_{\{\psi_j=c\}} \sum_{k=1} ^p \frac{\ind_{\{\psi_{k}=c\}} -\nu}{1-\nu} \sum_{k_1 = 1} ^p \Theta^{*,{(2)}} _{k k_1} \ind_{\{\pi (k_1) < \pi (k)\}} w_{\pi (k)}  \right) \\
  &= pq \E_{\pi} \left( \sum_{k=1} ^p X_{ik} \ind_{\{H_i = k\}} \sum_{k_1 = 1} ^p \Theta^{*,{(2)}} _{k k_1} \ind_{\{\pi (k_1) < \pi (k)\}} \sum_{\ell = 2} ^p w_{\ell} \ind_{\{\pi (k) = \ell\}} \right).
\end{align*}
where in the final line we have appealed to \eqref{eq:ind1}.
Now observe that for $k \in \mb{z}_i$,
\begin{align*}
 \ind_{\{H_i = k\}} \ind_{\{\pi(k_1)<\pi(k)\}} \ind_{\{\pi (k)=\ell\}} = \ind_{\{X_{ik_1} = 0\}} \ind_{\{H_i = k\}} \ind_{\{M_i = \ell, \, \pi(k_1) < \ell \}},
\end{align*}
and $\ind_{\{H_i = k\}}$ and $\ind_{\{M_i = \ell, \, \pi(k_1) < \ell \}}$ are independent. Thus we have
\begin{align*}
 \E_{\mbb{\pi}, \mbb{\Psi}}((\mb{S}\mb{b}^{*,(2)})_i) &= \sum_{k, k_1} X_{ik}\ind_{\{X_{ik_1} = 0\}}\Theta^{*,{(2)}} _{k k_1}  \sum_{\ell = 1} ^p p\pr_\pi(M_i = \ell, \pi(k_1)<\ell) w_{\ell} \\
 &= \sum_{k, k_1} X_{ik}\ind_{\{X_{ik_1} = 0\}}\Theta^{*,{(2)}} _{k k_1}  \sum_{\ell = 2} ^p (\ell-1) \pr_\pi(M_i = \ell |\, \pi(k_1)<\ell) w_{\ell} \\
 &= \sum_{k, k_1} X_{ik}\ind_{\{X_{ik_1} = 0\}}\Theta^{*,{(2)}} _{k k_1}  \sum_{\ell = 2} ^p (\ell-1) \frac{\binom{p-\ell}{q-1}}{\binom{p-1}{q}} w_{\ell}.
\end{align*}
Thus if we choose $\mb{w}$ such that
\begin{equation}
 \sum_{\ell = 2} ^p (\ell-1) \frac{\binom{p-\ell}{q-1}}{\binom{p-1}{q}} w_{\ell} = 1, \label{eq:constr_inter}
\end{equation}
property \eqref{eq:unbiased_inter} will be satisfied.

Next we compute
\begin{align}
\E(\|\mb b^{*,(2)}\|_2^2) &\leq \frac{p^2 q^2}{L(1-\nu)}\sum_{k=1}^p \E \bigg\{\bigg(\sum_{k_1 = 1} ^p \Theta^{*,{(2)}} _{k k_1} \ind_{\{\pi(k_1) < \pi(k)\}}\bigg)^2 w^2_{\pi(k)}\bigg\} \notag \\
&= \frac{p^2 q^2}{L(1-\nu)} \sum_{k=1}^p \sum_{\ell=1}^p w_\ell^2 \bigg( \sum_{k_1} (\Theta^{*,{(2)}} _{k k_1})^2 \pr(\pi(k)=\ell, \pi(k_1) < \ell ) \notag\\
& \qquad \qquad + \sum_{k_1 \neq k_2} \Theta^{*,{(2)}} _{k k_1}\Theta^{*,{(2)}} _{k k_2} \pr(\pi(k)=\ell, \pi(k_1)<\ell, \pi(k_2)<\ell) \bigg) \notag \\
&= \frac{p q^2}{L(1-\nu)} \sum_{k=1}^p \sum_{\ell=2}^p w_\ell^2 \bigg(\frac{\ell-1}{p-1} \sum_{k_1}(\Theta^{*,{(2)}} _{k k_1})^2 + \frac{(\ell-1)(\ell-2)}{(p-1)(p-2)}\sum_{k_1 \neq k_2} \Theta^{*,{(2)}} _{k k_1}\Theta^{*,{(2)}} _{k k_2} \bigg) \notag \\
&\leq \frac{pq^2}{(p-1) L(1-\nu)}  \sum_{k, k_1, k_2} \abs{\Theta^{*,{(2)}} _{k k_1} \Theta^{*,{(2)}} _{k k_2} } \sum_{\ell=2} ^p (\ell-1) w_{\ell} ^2. \label{eq:inter_bd_0}
\end{align}

Choosing
\begin{equation} \label{eq:def_W2}
 w_{\ell} = \frac{\binom{p-\ell}{q-1} \Big/ \binom{p-1}{q}}{ \sum_{\ell' = 2} ^p (\ell'-1) \left\{ \binom{p-\ell'}{q-1} \Big/ \binom{p-1}{q}\right\}^2}
\end{equation}
minimises \eqref{eq:inter_bd_0} subject to \eqref{eq:constr_inter} to give
\[
 \E_{\mbb{\pi}, \mbb{\Psi}} (\norms{\mb{b}^{*,(2)}}_2 ^2) \leq \frac{p q^2}{(p-1)L(1-\nu)}  \sum_{k, k_1, k_2} \abs{\Theta^{*,{(2)}} _{k k_1} \Theta^{*,{(2)}} _{k k_2} } \left\{ \sum_{\ell=1} ^{p-1} \ell \left( \frac{\binom{p-1-\ell}{q-1}}{\binom{p-1}{q}}\right)^2 \right\}^{-1}.
\]
Finally, Lemma~\ref{lem:inter_bd_lem} bounds the right-most term from above to yield
\begin{equation} \label{eq:inter_var2}
  \E_{\mbb{\pi}, \mbb{\Psi}} (\norms{\mb{b}^{*,(2)}}_2 ^2) \leq \frac{2\{(2-\delta)q\}^2}{L(1-\nu)}  \sum_{k, k_1, k_2} \abs{\Theta^{*,{(2)}} _{k k_1} \Theta^{*,{(2)}} _{k k_2} }.
\end{equation}

Now we turn to the mean-squared error. Observe that $\mb s_i^T \mb b^*$ is a sum of $L$ independent random variables, each having the same distribution as
\[
\sum_{c=1}^C S_{i1c} b^*_{1c} = \sum_{c=1}^C S_{i1c}\mb (b^{*,(1)}_{1c} + b^{*,(2)}_{1c}).
\]
Thus
\begin{align*}
\Var(\mb s_i^T \mb b^*) &\leq \frac{1}{L} \E \bigg(\sum_{c=1}^C S_{i1c}\mb (b^{*,(1)}_{1c} + b^{*,(2)}_{1c})\bigg)^2 \\
 &\leq \frac{1}{L} \Bigg[ \bigg\{\E \bigg(\sum_{c=1}^C S_{i1c}\mb b^{*,(1)}_{1c}\bigg)^2\bigg\}^{1/2} + \bigg\{\E \bigg(\sum_{c=1}^C S_{i1c}\mb b^{*,(2)}_{1c}\bigg)^2\bigg\}^{1/2} \Bigg]^2,
\end{align*}
where we have used the Cauchy--Schwarz inequality in the final line. Now using the fact that $\|\mb X\|_\infty \leq 1$, and following the argument that leads to \eqref{eq:equal_q_var}, we arrive at
\begin{align}
\E \bigg(\sum_{c=1}^C S_{i1c}\mb b^{*,(2)}_{1c}\bigg)^2 &=p^2q^2
\E\bigg\{\frac{\nu}{1-\nu}\sum_{k=1}^p\bigg( \sum_{k_1 = 1} ^p \Theta^{*,{(2)}} _{k k_1} \ind_{\{\pi (k_1) < \pi (k)\}} w_{\pi(k)} \bigg)^2 \notag \\
&\qquad\qquad + \frac{1-2\nu}{1-\nu}X_{iH_i}^2 \bigg( \sum_{k_1 = 1} ^p \Theta^{*,{(2)}} _{H_i k_1} \ind_{\{\pi (k_1) < M_i\}} w_{M_i} \bigg)^2\bigg\}. \label{eq:inter_bd_1}
\end{align}
We have
\begin{align}
\E \bigg\{  &  X_{iH_i}^2 \bigg( \sum_{k_1 = 1} ^p \Theta^{*,{(2)}} _{H_i k_1} \ind_{\{\pi (k_1) < M_i\}} w_{M_i} \bigg)^2\bigg\} =\frac{1}{q} \sum_{k=1}^p \sum_{\ell=1}^p X_{ik}^2 \E\bigg\{ \bigg(\sum_{k_1} \Theta^{*,(2)}_{k,k_1}  \ind_{\{\pi(k_1)<\ell\}} w_\ell \bigg)^2 \ind_{\{M_i=\ell\}} \bigg\} \notag\\
&= \sum_{k=1}^p X_{ik}^2 \sum_{\ell=2}^p w_\ell^2 \bigg( \sum_{k_1} (\Theta^{*,{(2)}} _{k k_1})^2 \pr(M_i=\ell, \pi(k_1) < \ell ) + \sum_{k_1 \neq k_2} \Theta^{*,{(2)}} _{k k_1}\Theta^{*,{(2)}} _{k k_2} \pr(M_i=\ell, \pi(k_1)<\ell, \pi(k_2)<\ell) \bigg) \notag\\
&= \sum_{k=1}^p X_{ik}^2 \sum_{\ell=2}^p w_\ell^2 \Bigg( \frac{\ell-1}{p-1} \frac{\binom{p-\ell}{q-1}}{\binom{p-1}{q}} \sum_{k_1} (\Theta^{*,{(2)}}_{k k_1})^2 + \frac{(\ell-1)(\ell-2)}{(p-1)(p-2)} \frac{\binom{p-\ell}{q-1}}{\binom{p-2}{q}} \sum_{k_1 \neq k_2} \Theta^{*,{(2)}} _{k k_1}\Theta^{*,{(2)}} _{k k_2} \Bigg). \label{eq:inter_bd_2}
\end{align}
Now
\[
\frac{\binom{p-\ell}{q-1}}{\binom{p-1}{q}} \leq \frac{q}{p-1} \qquad \text{and} \qquad \frac{\ell-2}{p-2} \frac{\binom{p-\ell}{q-1}}{\binom{p-2}{q}} \leq \frac{q}{p-1}.
\]
Thus by Lemma~\ref{lem:inter_bd_lem} the quantity in \eqref{eq:inter_bd_2} is at most
\[
 \frac{2(2-\delta)^2\delta}{p^2}\sum_{k,k_1,k_2}^p \abs{X_{ik}^2\Theta^{*,{(2)}} _{k k_1} \Theta^{*,{(2)}} _{k k_2} }.
\]
Returning to \eqref{eq:inter_bd_1} and using the argument leading to \eqref{eq:inter_bd_0} therefore gives us
\begin{align*}
\E \bigg(\sum_{c=1}^C S_{i1c}\mb b^{*,(2)}_{1c}\bigg)^2 \leq 2(2-\delta)^2q^2\bigg(\frac{\nu}{1-\nu} \sum_{k,k_1,k_2} \abs{\Theta^{*,{(2)}}_{k k_1} \Theta^{*,{(2)}}_{k k_2}} +  \delta \frac{1-2\nu}{1-\nu} \sum_{k,k_1,k_2} \abs{X_{ik}^2\Theta^{*,{(2)}}_{k k_1} \Theta^{*,{(2)}}_{k k_2}} \bigg),
\end{align*}
which then gives part (iii) of the result.
\qed

\section{Implications for the RKHS of the resemblance kernel} \label{sec:RKHS_interpret}
We first observe the following result that is an immediate consequence of \citet{BOUCHARD2013615}.
\begin{prop}
Consider the resemblance kernel with input space $\mathcal{X} = \{0,1\}^p$ and let $\mathcal{H}$ be the corresponding RKHS. Then $\mathcal{H}$ contains every function $f : \mathcal{X} \to \R$.
\end{prop}
\begin{proof}
Let $\mb X \in \R^{|\mathcal{X}| \times p}$ be the matrix with each row a different element of $\mathcal{X}$ and let $\mb K \in \R^{|\mathcal{X}| \times |\mathcal{X}|}$ be the matrix with $K_{\mb x \mb x'} = k(\mb x, \mb x')$ where $k$ is the resemblance kernel. \citet{BOUCHARD2013615} shows that $\mb K$ is positive definite. Given $f : \mathcal{X} \to \R$, let $\mb f \in \R^{|\mathcal{X}|}$ be the vector of function evaluations so $f_x = f(x)$. Let $\mb \alpha  = \mb K^{-1} \mb f$. Then
\[
f(\cdot) = \sum_{x \in \mathcal{X}} \alpha_x k( \cdot, x)
\]
so $f \in \mathcal{H}$.
\end{proof}

The following corollary of Theorem~\ref{thm:approx_error_inter} derives properties of the RKHS associated with the resemblance kernel from our approximation error bounds.
\begin{cor}
Let $\mathcal{H}$ be the RKHS of the resemblance kernel $k$ when the input space $\mathcal{X} \subset \{0,1\}^p$ is constrained such that every element has $q$ non-zeroes. Suppose $p \geq 3$. For $\mbb\theta^{(1)} \in \R^p$, $\mbb\Theta^{(2)} \in \R^{p \times p}$ and $\mbb\Theta = (\mbb\theta^{(1)}, \mbb\Theta^{(2)})$, define $f_{\mbb\Theta} : \mathcal{X} \to \R$ by
\[
f_{\mbb\Theta} (\mb x) = \sum_{k=1}^p x_k \theta^{(1)}_k + \sum_{k=1}^p \sum_{j=1}^p x_k(1-x_j) \Theta_{k, j}^{(2)}.
\]
Suppose $\mbb\Theta$ is such that $f_{\mbb\Theta}$ is centred so $\sum_{\mb x \in \mathcal{X}} f_{\mbb\Theta}(\mb x)=0$
Then $f_{\mbb\Theta} \in \mathcal{H}$ and $\|f_{\mbb\Theta}\|_\mathcal{H}^2 \leq (2-\delta)q\|\mbb\Theta\|^2$. In particular if $\mbb\Theta^{(2)} =\mb 0$ then $\|f_{\mbb\theta^{(1)}}\|_\mathcal{H}^2 \leq (2-\delta)q\|\mbb\theta^{(1)}\|^2_2$.
\end{cor}
\begin{proof}
Let $\mb K \in \R\in \R^{|\mathcal{X}| \times |\mathcal{X}|}$ be the matrix with $K_{x x'} = k(x, x')$. We will make use of the fact that $\mb K$ is positive definite \citep{BOUCHARD2013615}. Suppose $\mb X \in \{0,1\}^{|\mathcal{X}| \times p}$ has as each row a different element of $\mathcal{X}$. For $L \in \mathbb{N}$, let $\mb S_L$ be the matrix formed from $1$-bit min-wise hashing applied to $\mb X$ and let $\mb K_L = 2 \mb S_L \mb S_L^T /L - \mb J$ where $\mb J$ is a $|\mathcal{X}| \times |\mathcal{X}|$ matrix of 1's. Given $\mbb\Theta$, let $\mb b^*_L$ be as in the proof of Theorem~\ref{thm:approx_error_inter} (see Section~\ref{proof:thm:approx_error_inter}) constructed using the permutations and $\mbb\Psi$ matrix corresponding to $\mb S_L$.

Let $k_L:\mathcal{X} \times \mathcal{X} \to \R$ be the random kernel associated with $\mb K_L$, that is $k(x, x') = K_{L, xx'}$; further let $\mathcal{H}_L$ be the associated RKHS. Let $\tilde{\mb b}_L$ be a centred version of $\mb b^*_L$ so $\tilde{\mb b}_L = \mb b^*_L - \bar{\mb b^*}_L$. Observe that $\|\tilde{\mb b}_L\|_2^2 \leq \|\mb b^*_L \|_2^2$.
Let $f_L: \mathcal{X} \to \R$ be given by $f_L(x) = (S_L \tilde{\mb b}_L)_x$. Then $f_L \in \mathcal{H}_L$ and $\|f_L\|_{\mathcal{H}_L}^2 = L\|\tilde{\mb b}_L\|_2^2/2$.

Note that the construction of $\mb b^*_L$ ensures that each component block is i.i.d. Thus as $L \to \infty$, we have that almost surely
\[
L\|\tilde{\mb b}_L\| \leq L\|\mb b^*_L\|_2^2 \to L^2 \E \|(b^*_{L, 1}, b^*_{L, 2})^T\|_2^2 \leq 2(2-\delta)q\|\mbb\Theta\|^2
\]
(note that the expression on the right hand side of the limit does not in fact depend on $L$). Also $\mb K_L \to \mb K$ almost surely by the strong law of large numbers (see Section~\ref{sec:Jaccard}).

Now observe that as $\mb f$ is centred, $\|\mb S_L \tilde{\mb b}_L - \mb f\|_2^2 \leq \|\mb S_L \mb b^*_L - \mb f\|_2^2$. Thus from Theorem~\ref{thm:approx_error_inter} (iii) we have that $\mb S_L \tilde{\mb b}_L \to \mb f$ in probability where $\mb f \in \R^{|\mathcal{X}|}$ has components $f_x=f_{\mbb\Theta}(x)$. Therefore there exists a subsequence $L_j$ along which $\mb S_{L_j} \tilde{\mb b}_{L_j} \to \mb f$ almost surely. Thus, there exists a realisation of the random elements above such that simultaneously $\mb K_{L_j} \to \mb K$, $f_{L_j}(x) \to f(x)$ as $j \to \infty$ and $\lim_{j \to \infty} \|f_{L_j}\|_{\mathcal{H}_{L_j}}^2 \leq (2-\delta)q\|\mbb\Theta\|^2$. In particular we have that $\|f_{L_j}\|_{\mathcal{H}_{L_j}}$ is bounded for all $j$. Applying Lemma~\ref{lem:RKHS} then gives the result.
\end{proof}

\section{Empirical verification of Theorem~\ref{thm:approx_error_main_binary}}
In order to assess whether the scaling in $L$ provided by (iii) of Theorem~\ref{thm:approx_error_main_binary} is in line with what is observed in practice, we looked at several numerical experiments. We generated design matrices $\mb X \in \{0,1\}^{n \times p}$ with different levels of sparsity $q \in \{500, 1000, 5000\}$ and $(n, p)=(10^4, 10^5)$. Different $\mb S$ matrices were constructed for each of the three $\mb X$ matrices with $\log_2 L \in \{5, 6, \ldots, 12\}$. We then generated 100 vectors of coefficients $\mbb\beta^* \in \R^p$ with $\|\mbb\beta^*\|_2 = 1$ for each setting and examined
\begin{equation} \label{eq:empirical}
\text{err}_L := \min_{\mb b \in \R^L} \|\mb X \mbb\beta^* - \mb S \mb b\|_2^2/n.
\end{equation}
\begin{center}
\begin{figure}
\includegraphics[width=\textwidth]{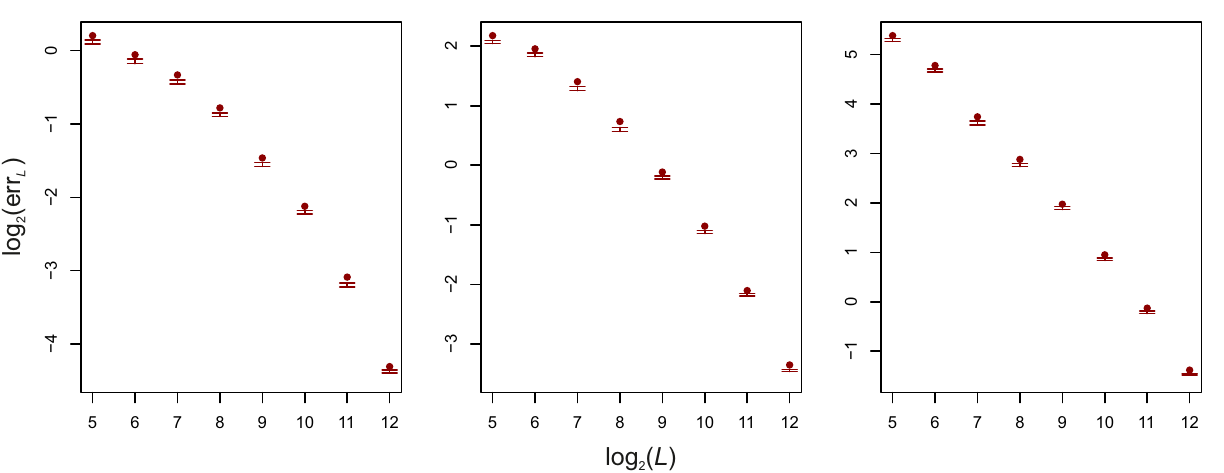}
\caption{\label{fig:plot}Plot of $\log_2(\text{err}_L)$ against $\log_2(L)$ for $q=500, 1000, 5000$ going from left to right. The bars give the first and third quartiles of $\log_2(\text{err}_L)$ over the 100 simulations, and the circles give maximum values.}
\end{figure}
\end{center}
Plots of $\log_2(\text{err}_L)$ against $\log_2(L)$ are given in Figure~\ref{fig:plot}. Given the scaling in $L$ suggested by Theorem~\ref{thm:approx_error_main_binary} (iii), we would expect the points to lie on a straight line with slope~$-1$.. We see this is indeed approximately the case for lager $L$ and $q$.

Note that Theorem~\ref{thm:approx_error_main_binary} does not make the claim that the $\mb b^*$ given is optimal in the sense of \eqref{eq:empirical}. Indeed it also satisfies unbiasedness and has a low $\ell_2$-norm in expectation: properties not necessarily satisfied by the minimiser of \eqref{eq:empirical}. Moreover, the bound must encompass a worst case in terms of $\mb X$ and the direction of $\mbb\beta^*$; tighter bounds may be used at the expense of a more complicated dependence on the precise form of $\mb X$ and $\mbb\beta^*$. However, we see from the empirical study that the scaling in $L$ provided by the result approximately parallels that corresponding to the minimiser of \eqref{eq:empirical}.

The details of the simulation study are as follows. The design $\mb X$ was generated randomly with the first $p/100$ columns containing $q/10$ 1's and the remaining columns containing $9q/10$ 1's. This mimics the setting of increasing variable sparsity described in Section~\ref{sec:OLS_main}. The vector of coefficients $\mbb\beta^*$ had its first $p/100$ entries generated independently with an $\text{Exp}(1)$ distribution and the remaining entries were set to 0; $\mbb\beta^*$ was then scaled to have $\ell_2$-norm 1.

\section{Prediction error results} \label{sec:pred_res}
Here we prove results for the prediction error under linear and logistic regression models. We denote the signal to be estimated by $\mb f^*$ and assume the existence of a $\mb b ^* \in \R^{2^b L}$ with
\begin{align*}
\frac{1}{n} \E(\|\mb f^* - \mb S \mb b^*\|_2^2) &\leq c_1/L \\
\E(\|\mb b^*\|_2^2) &\leq c_2/L. 
\end{align*}
Explicit constructions for such coefficient vectors are provided in the previous section. Using the results here in conjunction with the approximation error results proved in Section~\ref{app:approx} yield Theorems~\ref{thm:main_OLS_pred_error}--\ref{thm:interaction_logistic}: for example, substituting (iii) of Theorem~\ref{thm:approx_error_main_binary} immediately gives Theorem~\ref{thm:main_OLS_pred_error}.

\subsection{Linear regression}
We assume the model
\begin{equation}
\mb Y = \alpha^* \mb 1 + \mb f^* + \mbb\varepsilon,
\end{equation}
where $\Var(\mbb\varepsilon)=\sigma^2 \mb I$ and our goal is to estimate $\mb f^*$.
\begin{thm} \label{thm:OLS}
 Let $(\hat{\alpha}, \hat{\mb b})$ be the least squares estimator (\ref{eq:est_ols}). Then
 \begin{equation*}
  \MSPE ((\hat{\alpha}, \hat{\mb{b}})) \;\leq \; \frac{c_1}{L} + \frac{\sigma^2 \{(2^b-1)L+1\}}{n}.
 \end{equation*}
\end{thm}
\begin{proof}
Let us write
 \begin{equation*}
  \mb{Y} = \alpha^* \mb{1} + \mb f^* + \mbb{\varepsilon} = \alpha^* \mb{1} + \mb{S} \mb{b}^* + \mbb{\Delta} + \mbb{\varepsilon},
 \end{equation*}
so $\mbb{\Delta}$ is the approximation error of $\mb{S} \mb{b}^*$. Then we have
\begin{align*}
\MSPE ((\hat{\alpha}, \hat{\mb{b}})) = \frac{1}{n}\E_{\mbb{\varepsilon}, \mbb{\pi}, \mbb{\Psi}} (\|\alpha^* \mb{1}+ \mb f^* - \hat{\alpha} \mb{1} - \mb{S}\hat{\mb{b}}\|_2 ^2).
\end{align*}
Now let $\check{\mb{S}} = (\mb{1} \,\, \mb{S})$, and $\mb P_{\check{\mb{S}}}$ be the projection on to the column space of $\check{\mb{S}}$ (so $ \mb P_{\check{\mb{S}}} = \check{\mb{S}} \check{\mb{S}}^+$, where $\check{\mb{S}}^+$ denotes the Moore--Penrose pseudoinverse of $\check{\mb{S}}$).
 We have the following decomposition.
\begin{align*}
\alpha^* \mb{1}+ \mb f^* - \hat{\alpha} \mb{1} - \mb{S}\hat{\mb{b}} &= \alpha^* \mb{1} + \mb f^* - \mb P_{\check{\mb{S}}} \mb{Y} \\
 &= \alpha^* \mb{1} + \mb{S} \mb{b}^* + \mbb{\Delta} - \mb P_{\check{\mb{S}}} (\alpha^*\mb{1} + \mb{S} \mb{b}^* + \mbb{\Delta} + \mbb{\varepsilon}) \\
 &= (\mb I-\mb P_{\check{\mb{S}}})\mbb{\Delta} - \mb P_{\check{\mb{S}}} \mbb{\varepsilon}.
\end{align*}
Hence
\begin{align}
 \MSPE ((\hat{\alpha}, \hat{\mb{b}})) &= \frac{1}{n}\E_{\mbb{\varepsilon}, \mbb{\pi}, \mbb{\Psi}} (\|(\mb I-\mb P_{\check{\mb{S}}})\mbb{\Delta} - \mb P_{\check{\mb{S}}} \mbb{\varepsilon}\|_2 ^2) \notag \\
  &= \frac{1}{n} \E_{\mbb{\pi}, \mbb{\Psi}} (\|(\mb I- \mb P_{\check{\mb{S}}})\mbb{\Delta}\|_2^2) + \frac{1}{n} \E_{\mbb{\pi}, \mbb{\Psi}} \{ \E_{\mbb{\varepsilon}} ( \| \mb P_{\check{\mb{S}}} \mbb{\varepsilon}\|_2 ^2 \,|\, \mbb{\pi}, \mbb{\Psi}) \} \notag \\
 & \leq \frac{1}{n} \E_{\mbb{\pi}, \mbb{\Psi}} (\|\mbb{\Delta}\|_2 ^2) + \frac{\sigma^2 \{(2^b-1)L+1\}}{n} \label{eq:rank_P_S} \\
 & \leq \frac{c_1}{L} + \frac{\sigma^2 \{(2^b-1)L+1\}}{n}, \notag 
\end{align}
where in for \eqref{eq:rank_P_S} we have used the fact that $\text{rank}(\check{\mb S}) \leq (2^b-1)L+1$ as each the $L$ blocks sums to a vector of 1's
\end{proof}

\begin{thm} \label{thm:ridge}
There exists $\lambda$ depending on $\mb f^*$ and $\mb S$ such that defining
\begin{equation*}
  (\hat{\alpha}, \hat{\mb{b}}) := \argmin{(\alpha, \mb{b}) \in \R \times \R^L} \norms{\mb{Y} - \hat{\alpha} \mb{1} - \mb{S}\mb{b}}_2 ^2 \; \text{ such that } \; \norms{\mb{b}}_2 ^2 \leq \lambda,
\end{equation*}
we have
 \begin{equation*}
  \MSPE ((\hat{\alpha}, \hat{\mb{b}})) \;\leq \; \sigma\sqrt{\frac{c_2}{n}}+ \frac{c_1}{L} +\frac{\sigma^2}{n}.
 \end{equation*}
\end{thm}
\begin{proof}
We will take $\lambda = \|\mb b^*\|_2^2$.
Let a bar over any vector $\mb v$ denote the average of the components of $\mb v$, so $\bar{\mb v} = \sum_{j} v_j$. Note that
$ \hat{\alpha} = \overline{\mb Y - \mb S \hat{\mb b}}$, and define $\hat{a}^* = \overline{\mb Y - \mb S \mb b^*}$.
By our choice of $\lambda$, we have that
\begin{gather*}
  \norms{\mb{Y} - \hat{\alpha}\mb{1} - \mb{S}\hat{\mb{b}}}_2 ^2 \leq \norms{\mb{Y} - \hat{a}^*\mb{1} - \mb{S}\mb{b}^*}_2 ^2.
\end{gather*}
Noting that for any $\mb v, \mb u \in \R^n$, $\mb v^T (\mb u - \bar{\mb u} \mb{1}) = (\mb v - \bar{\mb v}\mb{1})^T \mb u$, rearranging the inequality above we get
\begin{gather} \label{eq:ridge_1}
 \norms{\alpha^*\mb 1 + \mb f^* - \hat{\alpha}\mb{1} - \mb{S}\hat{\mb{b}}}_2 ^2 \leq 2 (\mbb{\varepsilon} -\bar{\mbb{\varepsilon}}\mb{1})^T \mb{S} (\hat{\mb{b}} - \mb{b}^*) + \norms{\alpha^*\mb 1 + \mb f^* - \hat{a}^* \mb{1} - \mb{S}\mb{b}^*}_2 ^2.
\end{gather}
Now observe that
\begin{align}
 \norms{\alpha^*\mb 1 + \mb f^* - \hat{a}^* \mb{1} - \mb{S}\mb{b}^*}_2 ^2 &= \norms{\mb f^* - \overline{\mb f^*}\mb{1} - (\mb S\mb b^* -\overline{\mb S \mb b^*}\mb 1)}_2 ^2 + n\bar{\mbb\varepsilon}^2 \notag \\
 & \leq \norms{\mb f^* - \mb S\mb b^*}_2 ^2 + n\bar{\mbb\varepsilon}^2. \label{eq:bar_inequ1}
\end{align}
As $\mb{b}^*$ is independent of $\mbb{\varepsilon}$, taking expectations of \eqref{eq:ridge_1} yields
\begin{align} 
 \MSPE(\hat{\mb{b}}) =& \frac{2}{n}\E\{(\mbb{\varepsilon} - \bar{\mbb{\varepsilon}}\mb{1})^T \mb{S} \hat{\mb{b}}\} + \frac{1}{n}\E(\norms{\mb f^* - \mb{S}\mb{b}^*}_2 ^2) + \frac{\sigma^2}{n}. \label{eq:MSPE_ridge_1}
\end{align}
Now using the fact that $\|\hat{\mb b}\|_2 \leq \|\mb b^*\|_2$ and
applying the Cauchy--Schwarz inequality we have
\begin{align*}
 \E_{\mbb{\varepsilon}, \mbb{\pi}, \mbb{\Psi}} \{(\mbb{\varepsilon} - \bar{\mbb{\varepsilon}}\mb{1})^T \mb{S} \hat{\mb{b}}\} & \leq \sqrt{\E_{\mbb{\varepsilon}, \mbb{\pi}, \mbb{\Psi}} \{\norms{\mb{S}^T (\mbb{\varepsilon} - \bar{\mbb{\varepsilon}}\mb{1})}_2 ^2\}} \sqrt{\E (\norms{\mb{b}^*}_2 ^2)}.
\end{align*}
But
\begin{align*}
 \E_{\mbb{\varepsilon}} (\norms{\mb{S}^T (\mbb{\varepsilon} - \bar{\mbb{\varepsilon}}\mb{1})}_2 ^2 | \mbb{\pi}, \mbb{\Psi}) &= \E_{\mbb{\varepsilon}} [\Tr\{(\mbb{\varepsilon}-\bar{\mbb{\varepsilon}}\mb{1})^T \mb{S} \mb{S}^T (\mbb{\varepsilon}-\bar{\mbb{\varepsilon}}\mb{1})\}| \mbb{\pi}, \mbb{\Psi}] \\
 &= \E_{\mbb{\varepsilon}} [\Tr\{(\mbb{\varepsilon}-\bar{\mbb{\varepsilon}}\mb{1}) (\mbb{\varepsilon}-\bar{\mbb{\varepsilon}}\mb{1})^T \mb{S} \mb{S}^T \}| \mbb{\pi}, \mbb{\Psi}] \\
 &= \Tr[ \E_{\mbb{\varepsilon}}\{(\mbb{\varepsilon}-\bar{\mbb{\varepsilon}}\mb{1})( \mbb{\varepsilon}-\bar{\mbb{\varepsilon}}\mb{1})^T\} \mb{S} \mb{S}^T ] \\
 &= \sigma^2 \norms{(\mb I - n^{-1} \mb 1 \mb 1^T) \mb S}_F ^2 \leq \sigma^2\norms{\mb S}_F ^2 \leq \sigma^2 nL,
\end{align*}
whence
\begin{equation}
  \E_{\mbb{\varepsilon}, \mbb{\pi}, \mbb{\Psi}} \{(\mbb{\varepsilon}-\bar{\mbb{\varepsilon}}\mb{1})^T \mb{S} \hat{\mb{b}}\} \leq \sigma \sqrt{c_2 n}. \label{eq:b_hat_S_bd}
\end{equation}
Substituting in to \eqref{eq:MSPE_ridge_1} then gives the result.
\end{proof}

\subsection{Logistic regression}
We give an analogous result to Theorem~\ref{thm:main_ridge_pred_error} for classification problems under logistic loss. 
Let $\mb X \in [-1, 1]^{n \times p}$ be the design matrix of predictor variables and let $\mb Y \in \{0,1\}^n$ be an associated vector of class labels. We assume the model
\begin{equation*}
 Y_i \sim \mathrm{Bernoulli}(p_i); \qquad \log\left( \frac{p_i}{1 - p_i} \right) = f_i, 
\end{equation*}
with the $Y_i$ independent for $1 \leq i \leq n$.
Define
 \begin{equation*} 
  \hat{\mb{b}}_\lambda = \argmin{\mb{b}} \frac{1}{n}\sum_{i=1} ^n \left[ -Y_i \mb{s}_i ^T \mb{b} + \log\{1+\exp(\mb{s}_i ^T \mb{b})\} \right] \; \text{ such that } \; \norms{\mb{b}}_2 ^2 \leq \lambda.
 \end{equation*}
 Let $\mathcal{E}(\hat{\mb{b}}_\lambda)$ denote the excess risk of $\hat{\mb{b}}_\lambda$ under logistic loss, so
 \begin{equation*}
  \mathcal{E}(\hat{\mb{b}}_\lambda) = \frac{1}{n} \sum_{i=1} ^n \left[ -p_i \mb{s}_i ^T \hat{\mb{b}}_\lambda + \log\{1+\exp(\mb{s}_i ^T \hat{\mb{b}}_\lambda)\} \right] - \frac{1}{n} \sum_{i=1} ^n \left[ -p_i f_i + \log\{1+\exp(f_i)\} \right].
 \end{equation*}
\begin{thm} \label{thm:logistic}
 Let $\tilde{p} \in \R$  be given by \eqref{eq:def_p_tilde}.
Then we have that there exists $\lambda$ such that
\begin{equation*}
\E_{\mb{Y}, \mbb{\pi}, \mbb{\Psi}}\{\mathcal{E}(\hat{\mb{b}}_\lambda)\} \leq \frac{c_1}{4L} + \sqrt{\tilde{p}c_2/n}.
\end{equation*}
\end{thm}
\begin{proof}
We
take $\lambda=\|\mb b^*\|_2^2$. By the definition of $\hat{\mb{b}}$ (dropping the subscript $\lambda$), we have
\begin{gather*}
 \frac{1}{n} \sum_{i=1} ^n \left[ -Y_i \mb{s}_i ^T \hat{\mb{b}} + \log\{1+\exp(\mb{s}_i ^T \hat{\mb{b}})\} \right] \leq \frac{1}{n} \sum_{i=1} ^n \left[ -Y_i \mb{s}_i ^T \mb{b}^* + \log\{1+\exp(\mb{s}_i ^T \mb{b}^*)\} \right].
\end{gather*}
 Using this, analogously to \eqref{eq:ridge_1} we get, 
\begin{gather*}
 \mathcal{E}(\hat{\mb{b}}) \leq \frac{1}{n} \sum_{i=1} ^n (Y_i - p_i) \{\mb{S}(\hat{\mb{b}} - \mb{b}^*)\}_i + \mathcal{E}(\mb{b}^*).
\end{gather*}
Let $\mbb{\varepsilon}:=\mb{Y} - \mb{p}$ be the residual vector. Since $\mb b^*$ is independent of $\mbb\varepsilon$, after taking expectations we arrive at
\begin{align*}
 \E_{\mbb{\varepsilon}, \mbb{\pi}, \mbb{\Psi}}\{\mathcal{E}(\hat{\mb{b}})\} &\leq \frac{1}{n} \E_{\mbb{\varepsilon}, \mbb{\pi}, \mbb{\Psi}}(\mbb{\varepsilon} ^T \mb{S} \hat{\mb{b}}) + \E_{\mbb{\pi}, \mbb{\Psi}}\{\mathcal{E}(\mb{b}^*)\}.
\end{align*}
Write $h(a)=\log(1+e^{a})$. By the mean value theorem, we have
\begin{align*}
 |\mathcal{E}(\mb{b}^*)| &= \frac{1}{n}\sum_{i=1}^n|h(\mb s_i^T\mb b^*) - h(f_i) - (\mb s_i^T\mb b^* -f_i)h'(f_i)|\\
 &\leq \frac{1}{n}\sup_{a \in \R} h''(a) \norms{\mb f^* - \mb{S}\mb{b}^*}_2 ^2 \leq \frac{c_1}{4L}.
\end{align*}
The same argument that leads to \eqref{eq:b_hat_S_bd} gives
\begin{equation*}
 \frac{1}{n} \E_{\mbb{\varepsilon}, \mbb{\pi}, \mbb{\Psi}}(\mbb{\varepsilon} ^T \mb{S} \hat{\mb{b}}) \leq \frac{1}{n} \sqrt{\E_{\mbb{\varepsilon}, \mbb{\pi}, \mbb{\Psi}} (\norms{\mb{S}^T \mbb{\varepsilon}}_2 ^2)}  \sqrt{c_2/L} \leq \sqrt{\tilde{p}c_2/n}.
\end{equation*}
Collecting together the various inequalities, we get the required result.
\end{proof}
\section{Technical lemmas}
In this section we collect all technical lemmas used by the results presented earlier.
\begin{lem} \label{lem:sum_sq}
 Let $(a_i)_{i=1} ^\infty$ and $(b_i)_{i=1} ^\infty$ be two sequences of non-negative, non-increasing, real numbers such that that there is some $i^* \in \mathbb{N}$ for which
\begin{gather*}
 a_i \leq b_i \quad \text{for all } i \leq i^*, \\
 a_i \geq b_i \quad \text{for all } i > i^*.
\end{gather*}
\begin{enumerate}[(i)]
 \item If
 \begin{equation*}
  \sum_{i=1} ^\infty a_i = \sum_{i=1} ^\infty b_i < \infty,
 \end{equation*}
 and $m \geq 1$, then
 \begin{equation*}
  \sum_{i=1} ^\infty a_i ^m \leq \sum_{i=1} ^\infty b_i ^m.    
 \end{equation*}
\item If $(c_i)_{i=1} ^\infty$ is a sequence of non-negative, non-decreasing real numbers and
 \begin{equation*}
  \sum_{i=1} ^\infty b_i \leq \sum_{i=1} ^\infty a_i <\infty, \;\; \sum_{i=1} ^\infty c_i a_i \, , \, \sum_{i=1} ^\infty c_i b_i < \infty,
 \end{equation*}
 then
 \begin{equation*}
   \sum_{i=1} ^\infty c_i a_i \geq \sum_{i=1} ^\infty c_i b_i.
 \end{equation*}
\end{enumerate}
\end{lem}
\begin{proof}
Note that the sequence $(b_i)_{i=1} ^\infty$ majorises $(a_i)_{i=1} ^\infty$ (see page 191 of \citet{cauchy-schwarz}). 
Result (i) follows from applying Schur's majorisation inequality (\citet{cauchy-schwarz}; page 201) with the convex function $x \mapsto x^m$ on $[0, \infty)$.

 For (ii) we argue,
 \begin{equation*}
  \sum_{i=1} ^{i^*} c_i(b_i - a_i) \leq c_{i^*} \sum_{i=1} ^{i^*} (b_i - a_i) \leq c_{i^*} \sum_{i > i^*} (a_i - b_i) \leq \sum_{i > i^*} c_i (a_i - b_i). \qedhere
 \end{equation*}
\end{proof}

\begin{lem} \label{lem:inter_bd_lem}
Let $q, p \in \mathbb{N}$ with $q \geq 1$, $p \geq \max\{q, 3\}$. We have
 \begin{equation*}
  \sum_{\ell=1} ^{p-1} \ell \left( \frac{\binom{p-1-\ell}{q-1}}{ \binom{p-1}{q}} \right)^2 \geq \frac{1}{2(2-q/p)^2} \frac{p^2}{(p-1)^2}.
 \end{equation*}
\end{lem}
\begin{proof}
Let the sequences $(a_\ell)_{\ell=1} ^\infty$ and $(b_\ell)_{\ell=1} ^\infty$ be defined by
 \begin{align*}
  a_\ell = \begin{cases}
            \left( \frac{\binom{p-1-\ell}{q-1}}{ \binom{p-1}{q}} \right)^2  & \text{ if } 1 \leq \ell \leq p-1\\
            0 & \text{ otherwise},
           \end{cases}
\;\;\; b_\ell = \begin{cases}
                                                                                         \left(\frac{q}{p-1}\right)^2 & \text{ if } \ell \leq \floor{\frac{(p-1)^2}{\{2(p-1)-q\}q} } \\  \vspace{-8pt} \\
                                                                                         \frac{q}{2(p-1) -q} - \left(\frac{q}{p-1}\right)^2 \floor{\frac{(p-1)^2}{\{2(p-1)-q\}q} } & \text{ if } \ell = \floor{\frac{(p-1)^2}{\{2(p-1)-q\}q} } + 1 \\
                                                                                         0 & \text{ otherwise.}
                                                                                        \end{cases}
 \end{align*}
 Let the sequence $(c_\ell)_{\ell=1} ^\infty$ be defined by $c_\ell = \ell$.
Note the sequences $(a_\ell)_{\ell=1} ^\infty$, $(b_\ell)_{\ell=1} ^\infty$ and $(c_\ell)_{\ell=1} ^\infty$ satisfy the hypotheses of Lemma~\ref{lem:sum_sq}. Thus
\[
 \sum_{\ell=1} ^{p-1} \ell a_\ell \geq \sum_{\ell=1} ^{p-1} \ell b_\ell,
\]
and
\begin{align*}
 \sum_{\ell=1} ^{p-1} \ell b_\ell  & =\frac{1}{2}\left(\frac{q}{p-1}\right)^2 \left( \floor{\frac{(p-1)^2}{\{2(p-1)-q\}q} } + 1 \right) \floor{\frac{(p-1)^2}{\{2(p-1)-q\}q} } \\
 &\;\;\; + \left(\frac{q}{p-1}\right)^2 \left(\frac{(p-1)^2}{\{2(p-1)-q\}q} - \floor{\frac{(p-1)^2}{\{2(p-1)-q\}q} } \right) \left( \floor{\frac{(p-1)^2}{\{2(p-1)-q\}q} } + 1\right).
\end{align*}
Letting $x=(p-1)^2/ [\{2(p-1)-q\}q]$, we have
\begin{align*}
  \sum_{\ell=1} ^{p-1} \ell b_\ell  & = \frac{1}{2}(\floor{x}+1)\floor{x} + (x-\floor{x})(\floor{x} + 1) \\
  & = \frac{1}{2} x(x+1) - \frac{1}{2}\{(x-\floor{x})\floor{x} + (x-\floor{x})(x+1)\} + (x-\floor{x})(\floor{x}+1).
\end{align*}
Since $1 \geq 1/2 + (x-\floor{x})/2$, we see that
\[
 (x-\floor{x})(\floor{x} + 1) \geq \frac{1}{2}(x-\floor{x})(x+1+\floor{x}),
\]
so 
\begin{align*}
 \sum_{\ell=1}^{p-1} \ell b_\ell & \geq \frac{1}{2}x(x+1)\\
 &= \frac{1}{2}\left(\frac{(p-1)^2}{\{2(p-1)-q\}q} + 1\right) \frac{q}{2(p-1)-q} \\
 &= \frac{1}{2(p-1)} \frac{p + \{2-q/(p-1)\}q - 1}{\{2 - q/(p-1)\}^2} \\
 & \geq \frac{1}{2(p-1)} \frac{p+q}{\{2 - q/(p-1)\}^2} \\
 & \geq \frac{1}{2(2-q/p)^2} \frac{p^2}{(p-1)^2}. \qedhere
\end{align*}
\end{proof}

\begin{lem} \label{lem:taylor_coef_bd}
Let $\kappa(\delta)=\delta^{-a}$ where $a \in [0, 1]$. For $\ell \geq 2$,
\[
\abs{\frac{\kappa^{(\ell)}(1)}{\ell!}} \leq ae^a\frac{1}{\ell^{1-a}}.
\]
\end{lem}
\begin{proof}
\begin{align*}
\abs{\frac{\kappa^{(\ell)}(1)}{\ell!}} &= \frac{a(a+1)\cdots(a+\ell-1)}{1\cdot 2 \cdots \ell}\\
 &=\frac{a}{\ell} \frac{a+1}{1}\frac{a+2}{2} \cdots \frac{a+\ell-1}{\ell-1}.
\end{align*}
By Jensen's inequality
\begin{align*}
 \frac{1}{\ell-1}\bigg\{\log\bigg(\frac{a+1}{1}\bigg) &+ \log\bigg(\frac{a+2}{2}\bigg) +\cdots+ \log\bigg(\frac{a+\ell-1}{\ell-1}\bigg)\bigg\} \\
 & \leq \log\bigg(1 + \frac{a\{1+\log(\ell-1)\}}{\ell-1}\bigg),
\end{align*}
and
\[
 \bigg(1 + \frac{a\{1+\log(\ell-1)\}}{\ell-1}\bigg)^{\ell-1} \leq \exp[a\{1+\log(\ell-1)\}].
\]
Thus
\[
 \abs{\frac{\kappa^{(\ell)}(1)}{\ell!}} \leq a e^a \frac{(\ell-1)^a}{\ell} \leq ae^a\frac{1}{\ell^{1-a}}. \qedhere
\]
\end{proof}

\begin{lem} \label{lem:RKHS}
Suppose we have a sequence of positive definite kernels $\{k_L\}_{L=1}^\infty$ on a finite input space $\mathcal{X}$. For $L \in \mathbb{N}$, let $\mb K_L \in \R^{|\mathcal{X}| \times |\mathcal{X}|}$ be the matrix with $K_{L,x x'} = k_L(x, x')$. Suppose that $\mb K_L \to \mb K$ where $\mb K$ is positive definite and corresponds to kernel $k$. Let the RKHS's associated with $k_L$ and $k$ be $\mathcal{H}_L$ and $\mathcal{H}$ respectively. Suppose $f_L \in \mathcal{H}_L$ satisfies $|f_L(x) - f(x)| \to 0$ for some $f:\mathcal{X} \to \R$ and all $x \in \mathcal{X}$, and $\|f_L\|_{\mathcal{H}_L} < C$ for some $C > 0$. Then $f \in \mathcal{H}$ and $\|f_L\|_{\mathcal{H}_L} \to \|f \|_{\mathcal{H}}$ as $L \to \infty$.
\end{lem}
\begin{proof}
Since $\mathcal{X}$ is finite, for each $L$ there exists $\mbb\alpha_L \in \R^{\mathcal{|X|}}$ with $(f_L(x))_{x \in \mathcal{X}} = \mb K_L\mbb\alpha_L$. 
Writing $\mb f = (f(x))_{x \in \mathcal{X}}$, we have $\mb K_L \mbb\alpha_L \to \mb f$. Now $\mb f = \mb K \mbb\alpha$ where $\mbb\alpha = \mb K^{-1} \mb f$ showing that $f \in \mathcal{H}$.

It remains to show that $\mbb\alpha_L^T \mb K_L \mbb\alpha_L \to \mbb\alpha^T \mb K \mbb\alpha$. Note that $\mb K_L$ is positive definite for $L$ sufficiently large, so the fact that $\mbb\alpha_L^T \mb K_L \mbb\alpha_L < C$ ensures the $\mbb\alpha_L$ are bounded. Now suppose, for a contradiction, that there exists $\epsilon >0 $ and a subsequence $L_j$ with
\begin{equation} \label{eq:kern_lim}
|\mbb\alpha_{L_j}^T \mb K_{L_j} \mbb\alpha_{L_j} - \mbb\alpha^T \mb K \mbb\alpha|> \epsilon.
\end{equation}
Then as the $\mbb\alpha_{L_j}$ are bounded, there exists a further subsequence $L_{j_m} = l_m$ such that $\mbb\alpha_{l_m} \to \mbb\alpha_*$ as $m \to \infty$. But then since the fact that $\mb K_L \to \mb K$ implies the maximal eigenvalues of the $\mb K_L$ are bounded, $\mbb\alpha_{l_m}^T \mb K_{l_m} \mbb\alpha_{l_m} \to \mbb\alpha_*^T K_{l_m} \mbb\alpha_*$ as $m \to \infty$. But then $\mbb\alpha_{l_m}^T \mb K_{l_m} \mbb\alpha_{l_m} \to \mbb\alpha_{*}^T \mb K \mbb\alpha_{*}$, contradicting \eqref{eq:kern_lim}.
\end{proof}

\bibliographystyle{abbrvnat}
\bibliography{biball}

\end{document}